\documentclass[12pt, twoside]{article}
\usepackage{amsmath,amsthm,amssymb}
\usepackage{times}
\usepackage{enumerate}

\usepackage[latin1]{inputenc}
\usepackage{amsmath} 
\usepackage{amsfonts}
\usepackage{amssymb}
\usepackage{stmaryrd}
\usepackage{latexsym} 
\usepackage{graphicx}
\usepackage{subfigure}
\usepackage{color}
\usepackage{hyperref}
\usepackage{verbatim}
\usepackage[all]{xy}
\usepackage{graphics}
\usepackage{pdfsync}
\usepackage{tikz}
\usetikzlibrary{fit,calc,positioning,decorations.pathreplacing,matrix}
\usepackage{url}

\def\titlerunning#1{\gdef\titrun{#1}}
\makeatletter
\def\author#1{\gdef\autrun{\def\and{\unskip, }#1}\gdef\@author{#1}}
\def\address#1{{\def\and{\\\hspace*{18pt}}\renewcommand{\thefootnote}{}%
\footnote {#1}}%
\markboth{\autrun}{\titrun}}
\makeatother
\def\email#1{e-mail: #1}
\def\subjclass#1{{\renewcommand{\thefootnote}{}%
\footnote{\emph{Mathematics Subject Classification (2010):} #1}}}
\def\keywords#1{\par\medskip
\noindent\textbf{Keywords.} #1}

\newtheorem{theorem}{Theorem}[section]

\newtheorem{lemma}[theorem]{Lemma}

\theoremstyle{definition}
\newtheorem{definition}[theorem]{Definition}
\newtheorem{remark}[theorem]{Remark}
\newtheorem{example}[theorem]{Example}
\newtheorem{observation}[theorem]{Observation}

\numberwithin{equation}{section}

\newcommand{\bQ}{\mathbb{Q}}

\newcommand{\suchthat}{\;|\;}
\newcommand{\spam}{\operatorname{span}}

\newcommand{\sa}{{SA}}
\newcommand{\san}{{SA}_{n,n}}
\newcommand{\as}{{AS}}
\newcommand{\aseven}{{AS}_{n,n}}
\newcommand{\asodd}{{AS}_{n,n+1}}
\newcommand{\trt}{{TT}}
\newcommand{\trtn}{{TT}_{n,n,n}}
\newcommand{\ccf}{claw-contractible-free}
\newcommand{\cf}{claw-free}
\newcommand{\ep}{$e$-positive}

\newcommand{\coarsep}{\succcurlyeq _p}

\newlength\cellsize 
\savebox2{%
\begin{picture}(15,15)
\put(0,0){\line(1,0){15}}
\put(0,0){\line(0,1){15}}
\put(15,0){\line(0,1){15}}
\put(0,15){\line(1,0){15}}
\end{picture}}
\newcommand\cellify[1]{\def\thearg{#1}\def\nothing{}%
\ifx\thearg\nothing
\vrule width0pt height\cellsize depth0pt\else
\hbox to 0pt{\usebox2\hss}\fi%
\vbox to 15\unitlength{
\vss
\hbox to 15\unitlength{\hss$#1$\hss}
\vss}}
\newcommand\tableau[1]{\vtop{\let\\=\cr
\setlength\baselineskip{-16000pt}
\setlength\lineskiplimit{16000pt}
\setlength\lineskip{0pt}
\halign{&\cellify{##}\cr#1\crcr}}}
\savebox3{%
\begin{picture}(15,15)
\put(0,0){\line(1,0){15}}
\put(0,0){\line(0,1){15}}
\put(15,0){\line(0,1){15}}
\put(0,15){\line(1,0){15}}
\end{picture}}
\newcommand\expath[1]{%
\hbox to 0pt{\usebox3\hss}%
\vbox to 15\unitlength{
\vss
\hbox to 15\unitlength{\hss$#1$\hss}
\vss}}
\newcommand\bas[1]{\omit \vbox to \cellsize{ \vss \hbox to \cellsize{\hss$#1$\hss} \vss}}

\begin{document}


\baselineskip=17pt


\titlerunning{$e$-positivity of claw-contractible-free graphs}

\title{Resolving Stanley's $e$-positivity of claw-contractible-free graphs}

\author{Samantha Dahlberg
\and 
Ang\`{e}le Foley
\and
Stephanie van Willigenburg}

\date{}

\maketitle

\address{S. Dahlberg:  {
School of Mathematical and Statistical Sciences,
Arizona State University,
Tempe AZ 85287-1804, USA; \email{sdahlber@asu.edu}}
\and
A. Foley: Department of Physics and Computer Science, 
Wilfrid Laurier University, 
Waterloo ON N2L 3C5, Canada; \email{ahamel@wlu.ca}
\and
S. van Willigenburg: Department of Mathematics,
 University of British Columbia,
 Vancouver BC V6T 1Z2, Canada; \email{steph@math.ubc.ca} (corresponding author)}



\begin{abstract}
In Stanley's seminal 1995 paper on the chromatic symmetric function, he stated that there was no known graph that was not contractible to the claw and whose chromatic symmetric function was not $e$-positive, namely, not a positive linear combination of elementary symmetric functions. We resolve this by giving infinite families of graphs that are not contractible to the claw and whose chromatic symmetric functions are not $e$-positive. Moreover, one such family is additionally claw-free, thus establishing that the $e$-positivity of chromatic symmetric functions {is in general not dependent on the existence of an induced claw
or of a contraction to a claw.}

\end{abstract}

\section{Introduction}\label{sec:intro}
The chromatic polynomial is a classical graph invariant dating back to Birkhoff \cite{Birkhoff}, while symmetric functions date back even further to Cauchy \cite{Cauchy}.  Both the chromatic polynomial and symmetric functions continue to
see wide application in diverse areas such as algebraic geometry \cite{Harada, SW} and quantum
physics \cite{Chmutov, Klazar}. In 1995 Stanley \cite{Stan95} introduced for a finite simple graph $G$ a symmetric function generalization of the chromatic polynomial of $G$, known as the chromatic
symmetric function, $X_G$. Stanley \cite{Stan_AG} further defined symmetric function generalizations of the Tutte
polynomial of a graph (in the guise of the bad colouring polynomial) and the
chromatic polynomial of a hypergraph. Noble and Welsh \cite{Noble} introduced the
$U$-polynomial as a different symmetric function generalization of the Tutte polynomial to Stanley's, but showed them to be equivalent. Some time before Stanley's  chromatic symmetric function was conceived, Brylawski had introduced
the polychromate \cite{Bryl} as a multivariate generalization of the Tutte polynomial;
this was later shown by Sarmiento \cite{Sarm} to be equivalent to the $U$-polynomial and
hence to Stanley's symmetric function generalization of the Tutte polynomial
too. 

Regarding applications, the chromatic symmetric function has been applied to scheduling problems via a generalization of Breuer and Klivans \cite{BK},  while the symmetric bad colouring polynomial arose in the study of the Potts model by Klazar et al. \cite{Klazar}. The chromatic symmetric function has also been shown to distinguish various non-isomorphic trees, for example \cite{Jose2, MMW},   {which explore Stanley's fundamental} question of whether $X_G$ distinguishes non-isomorphic trees, \cite[p. 170]{Stan95}. Although $X_G$ does not satisfy a deletion-contraction recurrence, Orellana and Scott established a three-term deletion recurrence for $X_G$ when $G$ has a triangle \cite{Orellana}. This was generalized to a $k$-term deletion recurrence for $X_G$ when $G$ has a $k$-cycle by the first and third authors \cite{Lollipops}.

Other quasisymmetric function generalizations of Stanley's chromatic symmetric  function have been defined by Humpert \cite{Humpert} and Shareshian and Wachs \cite{SW},
the latter having been further studied, for example \cite{Ath, BrosnanC, Clearman}. The generalization defined by Shareshian and Wachs \cite{SW} was in part motivated
by Stanley's conjecture \cite[Conjecture 5.1]{Stan95} that if a poset is $(\mbox{\textbf{3}}+\mbox{\textbf{1}})$-free, then
the chromatic symmetric function of its incomparability graph is $e$-positive,
that is, a non-negative linear combination of elementary symmetric functions.
Stanley observes that this conjecture is equivalent to the Stanley-Stembridge poset chain conjecture \cite[Conjecture 5.5]{StanStem}. The incomparability graph of a
$(\mbox{\textbf{3}}+\mbox{\textbf{1}})$-free poset is claw-free, that is, does not contain $K_{1,3}$ as an induced
subgraph. After formulating his conjecture, Stanley asks whether it could not
be widened from incomparability graphs to a larger class of graphs, and is led to
venture that it may hold for graphs not contractible to a claw (allowing removal
of any parallel edges). Stanley establishes the truth of this wider conjecture for
paths and cycles \cite[Propositions 5.3 and 5.4]{Stan95}.

Gebhard and Sagan proved a number of special cases of the conjecture via generalizing the chromatic symmetric function from commuting variables to non-commuting variables \cite{GebSag}. Gasharov \cite{Gasharov} moved closer to the conjecture by proving the weaker statement that the chromatic symmetric  function of the incomparability graph of
a $(\mbox{\textbf{3}}+\mbox{\textbf{1}})$-free poset is $s$-positive, that is, a non-negative linear combination of
Schur functions.  { Gasharov's work} has led to its own avenue of research, for example \cite{Kaliz, SWW}.

In this paper we produce infinitely many graphs not contractible to a claw (even
regarding multiple edges of a contraction as a single edge) but which are not
$e$-positive, thereby resolving Stanley's problem 
\begin{quote} \emph{We don't know of a graph which is not contractible to $K_{13}$ (even regarding multiple edges of a contraction as a single edge) which is not $e$-positive}
\end{quote}
in the negative. The original
conjecture for incomparability graphs of $(\mbox{\textbf{3}}+\mbox{\textbf{1}})$-free posets remains open, as
 {does} Stanley's question as to precisely which graphs have $e$-positive 
chromatic symmetric function.

Our paper is structured as follows. In Section~\ref{sec:background} we recall relevant concepts that we will require later, and give the four graphs with the fewest number of vertices that are not contractible to the claw and whose chromatic symmetric functions are not $e$-positive in Figure~\ref{fig:4counter}.  {Section~\ref{sec:eightfamilies} exhibits eight infinite families of graphs that are distinguished by whether or not they are claw-free, are not contractible to the claw, or have a chromatic symmetric function that is $e$-positive. This result is Theorem~\ref{the:no_connection} and thus establishes that the $e$-positivity of the chromatic symmetric function of a graph is independent of whether it is claw-free or claw-contractible-free. A summary of this theorem can be found in Table~\ref{tab:eightfamilies}.} Then in Section~\ref{sec:saltires} we  generalize two of  {the graphs from Section~\ref{sec:background}} into two infinite families of graphs. First is the family of saltire graphs, $\sa _{a,b}$, where $a,b\geq 2$, which generalizes the graph from  {Section~\ref{sec:background}} with the fewest edges. We prove that these graphs are not contractible to the claw in Lemma~\ref{lem:sa_ccf}, and additionally for $n\geq 3$ we prove that $X_{\san}$ is not $e$-positive in Lemma~\ref{lem:sa_notepos}.

Having discovered a family of graphs with an even number of vertices that resolves Stanley's {problem} we then extend our results to graphs with any number of vertices via the second family of augmented saltire graphs, $\as _{a,b}$, where $a\geq2, b\geq 3$, which we also prove are not contractible to the claw in Lemma~\ref{lem:as_ccf}. For $n\geq 3$   we further prove that $X_{\aseven}$ and $X_{\asodd}$ are not $e$-positive in Lemma~\ref{lem:as_notepos}.

Finally, in Section~\ref{sec:towers} we introduce the family of triangular tower graphs, $\trt _{a,b,c}$, where $a,b,c\geq 2$, and prove in Lemma~\ref{lem:trt_ccf} that they are claw-free and do not contract to the claw, and for $n\geq3$ prove in Lemma~\ref{lem:trt_notepos} that $X_{\trtn}$ is not $e$-positive. This last example shows that the chromatic symmetric
 function of a graph need not be $e$-positive for graphs that have
neither an induced claw nor a claw obtained by contraction of edges (allowing
removal of parallel edges). This final family allows us to establish that there exists an infinite family of graphs that satisfies every combination of claw-free,  not being contractible to the claw and whose chromatic symmetric function is  $e$-positive.  In all cases we see that classical techniques  suffice to yield our proofs, though a number of technical lemmas such as Lemma~\ref{lem:coeff_power_G_n,n,n} are required to yield the final results.

\begin{table}
\begin{center}
\begin{tabular}{|c|c|c|c|}
\hline
claw-free&\ccf\ &$e$-positive&example graph family\\
\hline
\hline
yes&yes&yes&path, cycle, complete\\
\hline
yes&yes&no&triangular tower $\trtn$, $n\geq 3$\\
\hline
yes&no&yes&generalized bull\\
\hline
yes&no&no&generalized net\\
\hline
no&yes&yes&tadpole $T_{a,b}$, $a\geq4$\\
\hline
no&yes&no&saltire $\san$, augmented\\
&&& saltire $\aseven, \asodd$, $n\geq 3$\\
\hline
no&no&yes&spider $S(n,n-1,1)$, $n\geq 2$\\
\hline
no&no&no&star\\
\hline
\end{tabular}
\end{center}
\caption{Infinite graph families that satisfy every combination of claw-free, \ccf\ and $e$-positive.  {Definitions for specific graphs can be found in Sections~\ref{sec:eightfamilies}, \ref{sec:saltires} and \ref{sec:towers}.}}
\label{tab:eightfamilies}
\end{table}

\section{Background}\label{sec:background}
We begin by recalling some necessary combinatorial, algebraic and graph theoretic results that will be useful later. A \emph{partition} $\lambda = (\lambda_1, \lambda_2, \ldots , \lambda_{\ell(\lambda)})$ of $N$, denoted by $\lambda \vdash N$, is a list of positive integers whose \emph{parts} $\lambda_i$ satisfy $\lambda_1\geq \lambda _2\geq \cdots \geq\lambda_{\ell(\lambda)} >0$ and $\sum _{i=1}^{\ell(\lambda)} \lambda _i = N$.  If $\lambda$ has exactly $m_i$ parts equal to $i$ for $1\leq i \leq N$ we often denote $\lambda$ by $\lambda = (1^{m_1}, 2^{m_2}, \ldots , N^{m_N})$.

The algebra of symmetric functions is a subalgebra of $\bQ [[ x_1, x_2, \ldots ]]$ that can be defined as follows. The \emph{$i$-th elementary symmetric function} $e_i$ for $i\geq 1$ is given by
$$e_i = \sum _{j_1<j_2<\cdots < j_i} x_{j_1}x_{j_2}\cdots x_{j_i}$$and given a partition $\lambda = (\lambda _1, \lambda_2, \ldots , \lambda _{\ell(\lambda)})$  the \emph{elementary symmetric function} $e_\lambda$ is given by
$$e_\lambda = e_{\lambda _1}e_{\lambda _2}\cdots e_{\lambda _{\ell(\lambda)}}.$$The \emph{algebra of symmetric functions}, $\Lambda$, is then the graded algebra
$$\Lambda = \Lambda ^0 \oplus \Lambda ^1 \oplus \cdots$$where $\Lambda ^0 = \spam \{1\} = \bQ$ and for $N\geq 1$
$$\Lambda ^N = \spam \{e_\lambda \suchthat \lambda \vdash N\}.$$Moreover, the elementary symmetric functions form a basis for $\Lambda$ and if a symmetric function can be written as a non-negative linear combination of elementary symmetric functions, then we say it is \emph{$e$-positive}.

However, while the basis of elementary symmetric functions is central to the {problem} we wish to resolve, it is another basis, the basis of power sum symmetric functions, which will be central to our proofs. In terms of elementary symmetric functions the \emph{$i$-th power sum symmetric function} $p_i$ for $i\geq 1$ is given by
\begin{equation}
p_i= \sum_{\mu=(1^{m_1},2^{m_2},\ldots ,i^{m_i}) \vdash i}(-1)^{i-\ell(\mu)}\frac{i(\ell(\mu)-1)!}{\prod_{j=1}^{i}m_j!}e_{\mu}
\label{eq:newton}
\end{equation}and given a partition $\lambda = (\lambda _1, \lambda_2, \ldots , \lambda _{\ell(\lambda)})$  the \emph{power sum symmetric function} $p_\lambda$ is given by
$$p_\lambda = p_{\lambda _1}p_{\lambda _2}\cdots p_{\lambda _{\ell(\lambda)}}.$$This particular basis, of  power sum symmetric functions, will be useful later, as will the following. Given two partitions $\lambda, \mu \vdash N$ we write $\lambda \coarsep \mu$ if the parts of $\lambda$ are obtained by summing (not necessarily adjacent) parts of $\mu$. For example, $(5,4,2)\coarsep (3,3,2,2,1)$ since $5=3+2$, $4=3+1$ and $2=2$. Therefore by Equation~\eqref{eq:newton} we get the following key observation.

\begin{observation}\label{obs:newton}
When calculating the coefficient of $e_\mu$ in a symmetric function written in the basis of power sum symmetric functions, we need only focus on those $p_\lambda$ where $\lambda \coarsep \mu$.
\end{observation}

Since we will often want to compute the coefficient of a symmetric function $f\in \Lambda$ when written in the basis $\{b_\lambda \} _{\lambda \vdash N \geq 1}$, we will denote this by $[b_\lambda]f$. More details on these classical symmetric functions can be found in texts such as \cite{MacD, Sagan, ECII}, but for now we turn our attention to a more recent symmetric function, the chromatic symmetric function.

The chromatic symmetric function is reliant on a graph that is \emph{finite} and \emph{simple} and from now on we will assume that all our graphs satisfy these properties. This function is also reliant on all proper colourings of a graph. More precisely, given a graph, $G$, with vertex set $V$ a \emph{proper colouring} $\kappa$ of $G$ is a function
$$\kappa : V\rightarrow \{1,2,\ldots\}$$such that if  {$u, v \in V$} are adjacent, then  {$\kappa(u) \neq \kappa(v)$.} With this in mind we can now define the chromatic symmetric function, which we do in two ways before giving an example in Example~\ref{ex:claw}.

\begin{definition}\cite[Definition 2.1]{Stan95}\label{def:chromsym} For a graph $G$ with vertex set $V=\{v_1, v_2, \ldots, v_N\}$ and edge set $E$, the \emph{chromatic symmetric function} is defined to be
$$X_G = \sum _\kappa x_{\kappa(v_1)}x_{\kappa(v_2)}\cdots x_{\kappa(v_N)}$$
where the sum is over all proper colourings $\kappa$ of $G$. \end{definition}

{For brevity, we say a graph $G$ is \emph{$e$-positive} when $X_G$ is $e$-positive and \emph{not $e$-positive} otherwise.}

A more useful realisation of the chromatic symmetric function for us will be the following lemma, also due to Stanley, which requires some more notation. Given a graph $G$  with vertex set $V=\{v_1, v_2, \dots, v_N\}$, edge set $E$, and a subset $S\subseteq E$, let $\lambda(S)$ be the partition of $N$ whose parts are equal to the number of vertices in the connected components of the spanning subgraph of $G$ with vertex set $V$ and edge set $S$. If the number of vertices in a connected component is $\lambda _i$, then for succinctness we may refer to the connected component as a \emph{piece of size $\lambda_i$}.

\begin{lemma}\label{lem:stanley}\cite[Theorem 2.5]{Stan95} For a graph $G$ with vertex set $V$ and  edge set $E$ we have that
$$X_G=\sum_{S\subseteq E} (-1)^{|S|} p_{\lambda(S)}.$$
\end{lemma}

\begin{example}\label{ex:claw} The \emph{claw} (also known as $K_{1,3}$ or, as in Stanley's quote, $K_{13}$) shown 
 {in Figure~\ref{fig:claw}}
 has chromatic symmetric function
$$p_{(1^4)}-3p_{(2,1^2)}+3p_{(3,1)}-p_{(4)}=e_{(2,1^2)}-2e_{(2,2)}+5e_{(3,1)}+4e_{(4)}.$$
Thus, the claw is not $e$-positive.\end{example}

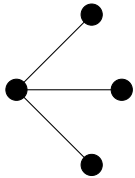
\begin{figure}[h]
\begin{center}
\begin{tikzpicture}[scale=1]
\coordinate (1) at (0,0);
\coordinate (2) at (1,1);
\coordinate (3) at (1.4,0);
\coordinate (4) at (1,-1);
\filldraw [black] (1) circle (4pt);
\filldraw [black] (2) circle (4pt);
\filldraw [black] (3) circle (4pt);
\filldraw [black] (4) circle (4pt);
\draw[->] (1)--(2);
\draw[->] (1)--(3);
\draw[->] (1)--(4);
\end{tikzpicture}
\end{center}
\caption{The claw.}
\label{fig:claw}
\end{figure}


While the focus of our paper will be on the claw, a number of well known graphs will also play a role: The  {\emph{complete graph}}, $K_N, N\geq 1$; the \emph{$N$-path}, $P_N, N\geq 1$; and the \emph{$N$-cycle}, $C_N, N\geq 3$, and we set $C_i=K_i$ for $i=1,2$. Two particular claw related properties of a graph will also be much in demand, the property of being claw-free and of being \ccf.

For the former recall that an \emph{induced} subgraph of a graph $G$ is a subgraph consisting of a subset of its vertices together with all edges whose endpoints both lie in the subset, and an \emph{edge induced} subgraph of a graph $G$ is a subgraph consisting of a subset of its edges together with all vertices at the endpoints of  {some} edge in the subset.

We say a graph is \emph{claw-free} if it does not have the claw as an induced subgraph. Claw-free graphs have the following characterization due to Beineke \cite[Theorem]{Beineke} who combined the results of Krausz \cite[Theorem 1]{Krausz} and van Rooij and Wilf \cite[Theorem 4]{vRW}.

\begin{lemma}\label{lem:beineke}\cite[Theorem]{Beineke}
A graph $G$ is claw-free  {if} there exists a partition of the edges of $G$ into disjoint sets, such that every set edge induces a complete subgraph of $G$ and no vertex  of $G$ belongs to more than 2 of these complete subgraphs.
\end{lemma}

For the latter, let $G$ be a graph with vertex set $V$ and edge set $E$. Recall a subset of $V$ is \emph{independent} if no two vertices in the subset are adjacent, plus when we delete a vertex $v$ from $G$ we delete  $v$ and all edges incident to $v$. Furthermore, if $S\subseteq E$, then we denote by $G/S$ the graph $G$ with the edges in $S$ contracted and the vertices at either end identified. We say that $G$ contracts to the claw if there exists $S\subseteq E$ such that $G/S$ yields the claw once multiple edges are replaced by single edges, and $G$ is \emph{\ccf}\ if $G$ does not contract to the claw. 

As with being claw-free, an elegant characterization exists for a graph to be \ccf. It is dependent on deleting  independent sets of vertices, and is a special case of a theorem by Brouwer and Veldman \cite[Theorem 3]{BV}.

\begin{lemma}\label{lem:BV}\cite[Theorem 3]{BV} A graph $G$ is \ccf\ if and only if the deletion of  {every} set of 3 independent vertices from $G$ results in a disconnected graph.
\end{lemma}

We now turn our attention to the graphs with the fewest number of vertices that are \ccf\ and not $e$-positive. Since a disconnected graph cannot contract to the claw, we restrict ourselves to \emph{connected} graphs in order to yield a meaningful resolution to Stanley's  {problem}. 

Otherwise,  by \cite[Propositon 2.3]{Stan95}, which says for disjoint graphs $G,H$ we have that
\begin{equation}\label{eq:XGXH}X_{G\cup H}=X_GX_H,\end{equation}calculating $X_{K_1}=e_1$, and Example~\ref{ex:claw}, we can conclude that the disjoint union of the claw and $K_1$ is a trivial resolution to Stanley's  {problem}.

Using an exhaustive computational search, the connected graphs with the fewest number of vertices, $N$, that are \ccf\ and not $e$-positive occur at $N=6$ and are given in Figure~\ref{fig:4counter}.

\begin{figure}[h]\label{fig:4counter}
\begin{center}
\begin{tikzpicture}[scale=1]
\begin{scope}[shift={(0,0)}]
\coordinate (1) at (0,0);
\coordinate (2) at (.8,1.25);
\coordinate (3) at (2.2,1.25);
\coordinate (4) at (3,0);
\coordinate (5) at (2.2,-1.25);
\coordinate (6) at  (.8,-1.25);
\filldraw [black] (1) circle (4pt);
\filldraw [black] (2) circle (4pt);
\filldraw [black] (3) circle (4pt);
\filldraw [black] (4) circle (4pt);
\filldraw [black] (5) circle (4pt);
\filldraw [black] (6) circle (4pt);
\draw[->] (2)--(5);
\draw[->] (3)--(6);
\draw[->] (1)--(2)--(3)--(4)--(5)--(6)--(1);
\end{scope}
\begin{scope}[shift={(4,0)}]
\coordinate (1) at (0,0);
\coordinate (2) at (.8,1.25);
\coordinate (3) at (2.2,1.25);
\coordinate (4) at (3,0);
\coordinate (5) at (2.2,-1.25);
\coordinate (6) at  (.8,-1.25);
\filldraw [black] (1) circle (4pt);
\filldraw [black] (2) circle (4pt);
\filldraw [black] (3) circle (4pt);
\filldraw [black] (4) circle (4pt);
\filldraw [black] (5) circle (4pt);
\filldraw [black] (6) circle (4pt);
\draw[->] (2)--(5);
\draw[->] (3)--(6);
\draw[->] (3)--(5);
\draw[->] (1)--(2)--(3)--(4)--(5)--(6)--(1);
\end{scope}
\begin{scope}[shift={(8,0)}]
\coordinate (1) at (0,0);
\coordinate (2) at (.8,1.25);
\coordinate (3) at (2.2,1.25);
\coordinate (4) at (3,0);
\coordinate (5) at (2.2,-1.25);
\coordinate (6) at  (.8,-1.25);
\filldraw [black] (1) circle (4pt);
\filldraw [black] (2) circle (4pt);
\filldraw [black] (3) circle (4pt);
\filldraw [black] (4) circle (4pt);
\filldraw [black] (5) circle (4pt);
\filldraw [black] (6) circle (4pt);
\draw[->] (2)--(5);
\draw[->] (3)--(6);
\draw[->] (1)--(4);
\draw[->] (1)--(2)--(3)--(4)--(5)--(6)--(1);
\end{scope}
\begin{scope}[shift={(12,0)}]
\coordinate (1) at (0,0);
\coordinate (2) at (.8,1.25);
\coordinate (3) at (2.2,1.25);
\coordinate (4) at (3,0);
\coordinate (5) at (2.2,-1.25);
\coordinate (6) at  (.8,-1.25);
\filldraw [black] (1) circle (4pt);
\filldraw [black] (2) circle (4pt);
\filldraw [black] (3) circle (4pt);
\filldraw [black] (4) circle (4pt);
\filldraw [black] (5) circle (4pt);
\filldraw [black] (6) circle (4pt);
\draw[->] (2)--(5);
\draw[->] (3)--(6);
\draw[->] (3)--(5);
\draw[->] (1)--(4);
\draw[->] (1)--(2)--(3)--(4)--(5)--(6)--(1);
\end{scope}
\end{tikzpicture}
\end{center}
\caption{From left to right the graphs $\sa _{3,3}, \as _{3,3}, K_{3,3}, AK_{3,3}$.}
\end{figure}
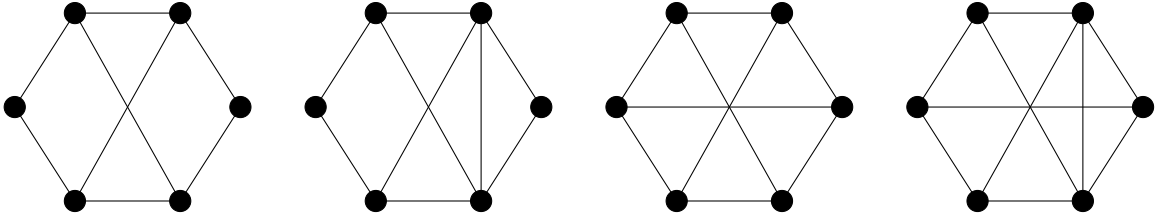

In particular, their chromatic symmetric functions are
$$\begin{array}{lclclclclcl}
X_{\sa _{3,3}}&=&2e_{(2,2,2)}&-&6e_{(3,3)}&+&26e_{(4,2)}&+&28e_{(5,1)}&+&102e_{(6)}\\
X_{\as _{3,3}}&=&2e_{(3,2,1)}&-&6e_{(3,3)}&+&24e_{(4,2)}&+&40e_{(5,1)}&+&120e_{(6)}\\
X_{K _{3,3}}&=&2e_{(2,2,2)}&-&12e_{(3,3)}&+&30e_{(4,2)}&+&24e_{(5,1)}&+&186e_{(6)}\\
X_{AK _{3,3}}&=&2e_{(3,2,1)}&-&6e_{(3,3)}&+&20e_{(4,2)}&+&32e_{(5,1)}&+&228e_{(6)}.
\end{array}$$

As we will see in  {Section~\ref{sec:saltires}}, the leftmost two graphs each naturally give rise to infinite families of graphs that are \ccf, are not $e$-positive and, moreover, we can explicitly identify and calculate a negative coefficient.  {However, we will first establish that $e$-positivity does not depend on the graph being either claw-free or claw-contractible-free.}

\section{Independence of $e$-positivity and the claw}\label{sec:eightfamilies} 
In order to prove that the $e$-positivity of a graph does not, in general, depend on the graph being claw-free nor \ccf\ we need to establish that there exists an infinite family of graphs that satisfies every combination of claw-free, \ccf, and $e$-positive. We achieve this in Theorem~\ref{the:no_connection}, in which we see that some of our families are also those that will resolve Stanley's problem later, and some are well-known families that we have already defined. The remaining families we define now, with an accompanying figure immediately following  {in Example~\ref{ex:other_graphs}.}

The \emph{generalized bull graph} $B_{a,b}$, where $a,b\geq 2$, is the graph on $a+b+1$ vertices created as follows. Consider $K_3$ with vertices $u, v, w$, and two paths $P_a$ and $P_b$. Then $B_{a,b}$ is created by identifying a leaf of $P_a$ with $u$ and identifying a leaf of $P_b$ with $v$. Similarly the \emph{generalized net graph} $GN_a$, where $a\geq 3$, is the graph on $a+3$ vertices created as follows. Consider $K_a$ and three copies of $P_2$. Then $GN_a$ is created by identifying a leaf of each $P_2$ with a distinct vertex of $K_a$. The \emph{tadpole graph} $T_{a,b}$, where $a\geq 3, b\geq 2$, is the graph on $a+b-1$ vertices created as follows. Consider $C_a$ and $P_b$. Then $T_{a,b}$ is created by identifying a leaf of $P_b$ with a vertex, denoted by $v$, of $C_a$.

Given a partition $\lambda  = (\lambda_1, \lambda_2, \ldots , \lambda_{\ell(\lambda)}) \vdash (N-1)$, where  ${\ell(\lambda)} \geq 3$, the \emph{spider graph} $S(\lambda_1, \lambda_2, \ldots , \lambda_{\ell(\lambda)})$ is the graph on $N$ vertices consisting of $P_{\lambda_1}, P_{\lambda_2}, \ldots , P_{\lambda_{\ell(\lambda)}}$ and a single vertex $v$ that is joined to a leaf in each $P_{\lambda_i}$ for $1\leq i \leq {\ell(\lambda)}$. Lastly, the \emph{star graph} $S_N$, where $N\geq 4$ is the spider graph $S(1^{N-1})$. 

\begin{example}\label{ex:other_graphs}  {Here are small examples for each of the types of graphs defined above.} Note that $S_4$ is the claw.\end{example}

\begin{figure}[h]
\begin{center}\begin{tikzpicture}[scale=1]
\begin{scope}[shift={(0,0)}]
\draw (0,-1.5) node {$B _{2,2}$};
\coordinate (1) at (-.5,0);
\coordinate (2) at (.5,0);
\coordinate (3) at (0,.5);
\coordinate (4) at (-1,0);
\coordinate (5) at (1,0);
\filldraw [black] (1) circle (2pt);
\filldraw [black] (2) circle (2pt);
\filldraw [black] (3) circle (2pt);
\filldraw [black] (4) circle (2pt);
\filldraw [black] (5) circle (2pt);
\draw[-] (1)--(2);
\draw[-] (2)--(3);
\draw[-] (1)--(3);
\draw[-] (1)--(4);
\draw[-] (2)--(5);
\end{scope}
\begin{scope}[shift={(4,0)}]
\draw (0,-1.5) node {$GN_4$};
\coordinate (1) at (-.5,0);
\coordinate (2) at (.5,0);
\coordinate (3) at (0,.5);
\coordinate (4) at (-1,0);
\coordinate (5) at (1,0);
\coordinate (6) at (0,1);
\coordinate (7) at (0,-.5);
\filldraw [black] (1) circle (2pt);
\filldraw [black] (2) circle (2pt);
\filldraw [black] (3) circle (2pt);
\filldraw [black] (4) circle (2pt);
\filldraw [black] (5) circle (2pt);
\filldraw [black] (6) circle (2pt);
\filldraw [black] (7) circle (2pt);
\draw[-] (1)--(2);
\draw[-] (2)--(3);
\draw[-] (1)--(3);
\draw[-] (1)--(4);
\draw[-] (2)--(5);
\draw[-] (3)--(6);
\draw[-] (1)--(7);
\draw[-] (2)--(7);
\draw[-] (3)--(7);
\end{scope}
\begin{scope}[shift={(8,0)}]
\draw (0,-1.5) node {$T_{4,2}$};
\coordinate (1) at (0,.5);
\coordinate (2) at (-.5,0);
\coordinate (3) at (.5,0);
\coordinate (4) at (0,-.5);
\coordinate (5) at (0,1);
\filldraw [black] (1) circle (2pt);
\filldraw [black] (2) circle (2pt);
\filldraw [black] (3) circle (2pt);
\filldraw [black] (4) circle (2pt);
\filldraw [black] (5) circle (2pt);
\draw[-] (1)--(2);
\draw[-] (1)--(3);
\draw[-] (4)--(3);
\draw[-] (4)--(2);
\draw[-] (5)--(1);
\end{scope}
\begin{scope}[shift={(11,0)}]
\draw (.75,-1.5) node {$S(2,1,1)$};
\coordinate (1) at (0,0);
\coordinate (2) at (.5,0);
\coordinate (3) at (1,0);
\coordinate (4) at (1.5,0);
\coordinate (5) at (.5,.5);
\filldraw [black] (1) circle (2pt);
\filldraw [black] (2) circle (2pt);
\filldraw [black] (3) circle (2pt);
\filldraw [black] (4) circle (2pt);
\filldraw [black] (5) circle (2pt);
\draw[-] (1)--(2)--(3)--(4);
\draw[-] (2)--(5);
\end{scope}
\end{tikzpicture}
\end{center}
\end{figure}

\begin{theorem}
\label{the:no_connection}
There exists an infinite family of graphs that satisfies every combination of claw-free, \ccf\ and $e$-positive. 
\end{theorem}

\begin{proof} We will provide infinite families of graphs that satisfy all eight combinations of claw-free, \ccf, and $e$-positive. A summary is provided in Table~\ref{tab:eightfamilies}.

\emph{Claw-free, \ccf, \ep:} The complete graphs $K_N$ for $N\geq 1$, the $N$-paths $P_N$ for $N\geq 1$, and the $N$-cycles $C_N$ for $N\geq 1$ are straightforwardly claw-free and \ccf. The graphs $P_N$ and $C_N$ were proved explicitly to be \ep\ by Stanley \cite[Propositions 5.3 and 5.4]{Stan95}, and $K_N$ is \ep\ since $X_{K_N}=N! e_N$.

\emph{Claw-free, \ccf, not \ep:} The triangular tower graphs $\trtn$ for $n\geq 3$ defined in Section~\ref{sec:towers} and illustrated in Figure~\ref{fig:tt} are proved to satisfy these three properties in Theorem~\ref{the:trtnotpositive}.

\emph{Claw-free, not \ccf, \ep:} The generalized bull graphs $B_{a,b}$ for $a,b \geq 2$ are claw-free by Lemma~\ref{lem:beineke} because of the following edge partition: The edges in the $K_3$ form one set, and every remaining edge forms its own set. These graphs contract to the claw by Lemma~\ref{lem:BV} because we can delete three independent vertices  {(the vertex $w$, and the two degree one vertices) resulting in a connected graph.} The generalized bull graphs are \ep\ as they are a special case of a corollary by Gebhard and Sagan \cite[Corollary 7.7]{GebSag}. 

\emph{Claw-free, not \ccf, not \ep:} The generalized net graphs $GN_a$ for $a\geq 3$ are \cf\ by Lemma~\ref{lem:beineke} because of the following edge partition: The edges in the $K_a$ form one set, and each of the three remaining edges forms its own set. These graphs contract to the claw by Lemma~\ref{lem:BV} because we can  delete the three independent degree one vertices resulting in a connected graph, and were proved to be  {not} \ep\ by Foley  {et al.} \cite[Theorem 1]{SpiderKin}.

\emph{Not \cf, \ccf, \ep:} Consider the tadpole graphs $T_{a,b}$ for $a\geq 4$, $b\geq 2$, that is, all tadpole graphs where $a\neq 3$. These graphs are not \cf\ due to the induced claw consisting of $v$ and the three vertices adjacent to it, since $a\neq 3$. They are also \ccf\ by Lemma~\ref{lem:BV}, as if we delete three independent vertices then at least two vertices belong to the copy of $C_a$ or $P_b$, resulting in a disconnected graph. All tadpole graphs are \ep\ by combining the results of Gebhard and Sagan \cite[Proposition 6.8]{GebSag} and \cite[Theorem 7.6]{GebSag} repeatedly.

\emph{Not \cf, \ccf, not \ep:} The saltire graphs $\san$ for $n\geq3$ and augmented saltire graphs $\aseven$ and $\asodd$ for $n\geq 3$ defined in Section~\ref{sec:saltires} and illustrated in Figures~\ref{fig:saltires} and \ref{fig:augsaltires} are proved to be \ccf\ and not \ep\ in Theorems~\ref{the:sanotpositive} and \ref{the:asnotpositive}. These graphs are not \cf\ due to the induced claw consisting of $v_2$ and the three vertices adjacent to it.

\emph{Not \cf, not \ccf,  \ep:} The spider graphs $S(n, n-1, 1)$ for $n\geq 2$ satisfy these properties. Firstly, they are not \cf\ due to the induced claw consisting of $v$ and the three vertices adjacent to it. Secondly, they contract to the claw by Lemma~\ref{lem:BV} since deleting the three vertices of degree one, which are independent, results in a connected graph. Now we will prove these graphs are \ep, and for this we need to recall the following by Orellana and Scott \cite[Theorem 3.1]{Orellana}: If $G$ is a graph with edge set $E$ such that $\epsilon_1, \epsilon_2, \epsilon_3 \in E$ form an edge induced $K_3$, then
\begin{equation}\label{eq:trip-del}
X_G = X_{G- \{\epsilon_1\}}+X_{G- \{\epsilon_2\}}- X_{G- \{\epsilon_1, \epsilon_2\}}
\end{equation}where $G-S$ for some $S\subseteq E$ is the  {graph with vertex set $V$ and edge set $E-S$.}

Consider the generalized bull graph $B_{n,n}$ for $n\geq 2$. Let $\epsilon_1, \epsilon_2, \epsilon_3$ be the three edges in its $K_3$ such that $\epsilon_1, \epsilon_2$ are incident to $w$ and $\epsilon_3$ is the edge $uv$. By Equations~\eqref{eq:XGXH} and \eqref{eq:trip-del} we have that
$$X_{B_{n,n}}=2X_{S(n,n-1,1)}-X_{P_{2n}}X_{K_1}$$
so $X_{S(n,n-1,1)}=\frac{1}{2}(X_{B_{n,n}}+X_{P_{2n}}X_{K_1})$. We know from above that $B_{n,n}$, $P_{2n}$ and $K_1$ are all \ep, and hence so is $S(n,n-1,1)$.

\emph{Not \cf, not \ccf, not \ep:} The star graphs $S_N$ for $N\geq4$ are not \cf\ due to the induced claw consisting of $v$ and any three vertices adjacent to it. These graphs contract to the claw by Lemma~\ref{lem:BV} since deleting any three vertices of degree one, which are independent, results in a connected graph. Lastly, these graphs were shown to be not \ep\ by Dahlberg, She and van Willigenburg \cite[Example 11]{Trees}. \end{proof}

\section{Saltire and augmented saltire graphs}\label{sec:saltires}
In this section we begin by introducing our first infinite family of graphs to resolve Stanley's  {problem}. In particular, this family includes the graph with the fewest number of vertices and edges, as verified by computer, which is \ccf\ and yet whose chromatic symmetric function is not $e$-positive.

The \emph{saltire graph} $\sa _{a,b}$, where $a,b \geq 2$, is the graph on $a+b$  vertices $\{v_1, v_2, \ldots , v_{a+b}\}$ with edges $v_iv_{i+1}$ for $1\leq i \leq a+b-1$, $v_{a+b}v_1$, $v_1v_{a+1}$ and $v_2v_{a+2}$. For example, $\sa _{2,2}= K_4$, and $\sa _{3,3}$ and a graphical representation of a generic $\sa _{a,b}$ are given in Figure~\ref{fig:saltires}.

\pagebreak

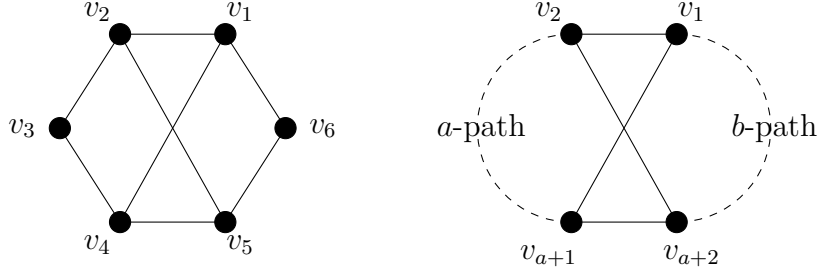
\begin{figure}[h]
\begin{center}
\begin{tikzpicture}[scale=1]
\begin{scope}[shift={(0,0)}]
\coordinate (1) at (0,0);
\coordinate (2) at (.8,1.25);
\coordinate (3) at (2.2,1.25);
\coordinate (4) at (3,0);
\coordinate (5) at (2.2,-1.25);
\coordinate (6) at  (.8,-1.25);
\draw (-.5,0) node {$v_3$};
\draw (.5,1.55) node {$v_2$};
\draw (2.4,1.55) node {$v_1$};
\draw (3.5,0) node {$v_6$};
\draw (2.4,-1.55) node {$v_5$};
\draw (.5,-1.55) node {$v_4$};
\filldraw [black] (1) circle (4pt);
\filldraw [black] (2) circle (4pt);
\filldraw [black] (3) circle (4pt);
\filldraw [black] (4) circle (4pt);
\filldraw [black] (5) circle (4pt);
\filldraw [black] (6) circle (4pt);
\draw[->] (2)--(5);
\draw[->] (3)--(6);
\draw[->] (1)--(2)--(3)--(4)--(5)--(6)--(1);
\end{scope}
\begin{scope}[shift={(6,0)}]
\coordinate (1) at (0,0);
\coordinate (2) at (.8,1.25);
\coordinate (3) at (2.2,1.25);
\coordinate (4) at (3,0);
\coordinate (5) at (2.2,-1.25);
\coordinate (6) at  (.8,-1.25);
\draw[dashed] (2) arc (90:270:1.25);
\draw[dashed] (3) arc (90:-90:1.25);
\draw (-.4,0) node {$a$-path};
\draw (3.5,0) node {$b$-path};
\draw (.5,1.55) node {$v_2$};
\draw (2.4,1.55) node {$v_1$};
\draw (2.4,-1.7) node {$v_{a+2}$};
\draw (.5,-1.7) node {$v_{a+1}$};
\filldraw [black] (2) circle (4pt);
\filldraw [black] (3) circle (4pt);
\filldraw [black] (5) circle (4pt);
\filldraw [black] (6) circle (4pt);
\draw[->] (2)--(5);
\draw[->] (3)--(6);
\draw[->] (2)--(3);
\draw[->] (5)--(6);
\end{scope}
\end{tikzpicture}
\end{center}
\caption{From left to right we have $\sa _{3,3}$ and a generic $\sa_{a,b}$.}
\label{fig:saltires}
\end{figure}

From the graphical representation we see that the edges of $\sa _{a,b}$ can be naturally partitioned into three parts as follows. Given $\sa _{a,b}$ we refer to the subgraph induced by the edges $$\{ v_iv_{i+1} \suchthat 2\leq i \leq a\}$$as the \emph{$a$-path}, the subgraph induced by the edges $$\{ v_iv_{i+1} \suchthat a+2\leq i \leq a+b-1\}\cup\{v_{a+b}v_1\}$$as the \emph{$b$-path}, and the subgraph induced by the edges $$\{ v_1v_2, v_1v_{a+1}, v_2v_{a+2}, v_{a+1}v_{a+2}\}$$as the \emph{middle}. Furthermore, when considering $\san$ we refer to the $a$-path, where $a=n$, as the \emph{left $n$-path}, and the $b$-path, where $b=n$, as the \emph{right $n$-path}, to distinguish them. With these definitions in hand, we come to our first result on saltire graphs.

\begin{lemma}\label{lem:sa_ccf} For all $a,b \geq 2$ the graph $\sa _{a,b}$ is \ccf. In particular, for $n\geq 3$ the graph $\san$ is \ccf.
\end{lemma}

\begin{proof} By Lemma~\ref{lem:BV} it suffices to show that the deletion of any three independent vertices from $\sa _{a,b}$ results in a disconnected graph. The pigeonhole principle guarantees that at least two of these independent vertices will belong to either the $a$-path or the $b$-path, and we can see from Figure~\ref{fig:saltires} that the removal of any two non-adjacent vertices from the $a$-path or the $b$-path results in a disconnected graph.
\end{proof}

Now that we have proved that $\sa _{a,b}$ is \ccf, we will restrict our attention to $\san$ where $n\geq 3$, and prove that its chromatic symmetric function is  not $e$-positive by calculating the coefficient $[e_{(n^2)}]X_{\san}$. Note that since $(2n)$ and $(n^2)$ are the only partitions $\lambda \vdash 2n$ satisfying $\lambda \coarsep (n^2)$ by Observation~\ref{obs:newton} and Lemma~\ref{lem:stanley}, in order to calculate $[e_{(n^2)}]X_{\san}$ we need to calculate $[p_{(2n)}]X_{\san}$ and $[p_{(n^2)}]X_{\san}$.


\begin{table}
\begin{center}
\begin{tabular}{| c |c|}
\hline
$i$& $\sa_{n,n}$  with $i$ middle edges removed\\
\hline
1&
\begin{tikzpicture}[scale=.5]
\coordinate (1) at (0,0);
\coordinate (2) at (.8,1.25);
\coordinate (3) at (2.2,1.25);
\coordinate (4) at (3,0);
\coordinate (5) at (2.2,-1.25);
\coordinate (6) at  (.8,-1.25);
\draw (2) arc (90:270:1.25);
\draw (3) arc (90:-90:1.25);
\draw (.5,1.55) node {$v_2$};
\draw (2.4,1.55) node {$v_1$};
\draw (2.4,-1.7) node {$v_{n+2}$};
\draw (.5,-1.7) node {$v_{n+1}$};
\filldraw [black] (2) circle (4pt);
\filldraw [black] (3) circle (4pt);
\filldraw [black] (5) circle (4pt);
\filldraw [black] (6) circle (4pt);
\draw[->] (3)--(6);
\draw[->] (2)--(3);
\draw[->] (5)--(6);
\end{tikzpicture}
\begin{tikzpicture}[scale=.5]
\coordinate (1) at (0,0);
\coordinate (2) at (.8,1.25);
\coordinate (3) at (2.2,1.25);
\coordinate (4) at (3,0);
\coordinate (5) at (2.2,-1.25);
\coordinate (6) at  (.8,-1.25);
\draw (2) arc (90:270:1.25);
\draw (3) arc (90:-90:1.25);
\draw (.5,1.55) node {$v_2$};
\draw (2.4,1.55) node {$v_1$};
\draw (2.4,-1.7) node {$v_{n+2}$};
\draw (.5,-1.7) node {$v_{n+1}$};
\filldraw [black] (2) circle (4pt);
\filldraw [black] (3) circle (4pt);
\filldraw [black] (5) circle (4pt);
\filldraw [black] (6) circle (4pt);
\draw[->] (2)--(5);
\draw[->] (2)--(3);
\draw[->] (5)--(6);
\end{tikzpicture}
\begin{tikzpicture}[scale=.5]
\coordinate (1) at (0,0);
\coordinate (2) at (.8,1.25);
\coordinate (3) at (2.2,1.25);
\coordinate (4) at (3,0);
\coordinate (5) at (2.2,-1.25);
\coordinate (6) at  (.8,-1.25);
\draw (2) arc (90:270:1.25);
\draw (3) arc (90:-90:1.25);
\draw (.5,1.55) node {$v_2$};
\draw (2.4,1.55) node {$v_1$};
\draw (2.4,-1.7) node {$v_{n+2}$};
\draw (.5,-1.7) node {$v_{n+1}$};
\filldraw [black] (2) circle (4pt);
\filldraw [black] (3) circle (4pt);
\filldraw [black] (5) circle (4pt);
\filldraw [black] (6) circle (4pt);
\draw[->] (2)--(5);
\draw[->] (3)--(6);
\draw[->] (5)--(6);
\end{tikzpicture}
\begin{tikzpicture}[scale=.5]
\coordinate (1) at (0,0);
\coordinate (2) at (.8,1.25);
\coordinate (3) at (2.2,1.25);
\coordinate (4) at (3,0);
\coordinate (5) at (2.2,-1.25);
\coordinate (6) at  (.8,-1.25);
\draw (2) arc (90:270:1.25);
\draw (3) arc (90:-90:1.25);
\draw (.5,1.55) node {$v_2$};
\draw (2.4,1.55) node {$v_1$};
\draw (2.4,-1.7) node {$v_{n+2}$};
\draw (.5,-1.7) node {$v_{n+1}$};
\filldraw [black] (2) circle (4pt);
\filldraw [black] (3) circle (4pt);
\filldraw [black] (5) circle (4pt);
\filldraw [black] (6) circle (4pt);
\draw[->] (2)--(5);
\draw[->] (3)--(6);
\draw[->] (2)--(3);
\end{tikzpicture}\\
\hline
2& 
\begin{tikzpicture}[scale=.5]
\coordinate (1) at (0,0);
\coordinate (2) at (.8,1.25);
\coordinate (3) at (2.2,1.25);
\coordinate (4) at (3,0);
\coordinate (5) at (2.2,-1.25);
\coordinate (6) at  (.8,-1.25);
\draw (2) arc (90:270:1.25);
\draw (3) arc (90:-90:1.25);
\draw (.5,1.55) node {$v_2$};
\draw (2.4,1.55) node {$v_1$};
\draw (2.4,-1.7) node {$v_{n+2}$};
\draw (.5,-1.7) node {$v_{n+1}$};
\filldraw [black] (2) circle (4pt);
\filldraw [black] (3) circle (4pt);
\filldraw [black] (5) circle (4pt);
\filldraw [black] (6) circle (4pt);
\draw[->] (2)--(3);
\draw[->] (5)--(6);
\end{tikzpicture}
\begin{tikzpicture}[scale=.5]
\coordinate (1) at (0,0);
\coordinate (2) at (.8,1.25);
\coordinate (3) at (2.2,1.25);
\coordinate (4) at (3,0);
\coordinate (5) at (2.2,-1.25);
\coordinate (6) at  (.8,-1.25);
\draw (2) arc (90:270:1.25);
\draw (3) arc (90:-90:1.25);
\draw (.5,1.55) node {$v_2$};
\draw (2.4,1.55) node {$v_1$};
\draw (2.4,-1.7) node {$v_{n+2}$};
\draw (.5,-1.7) node {$v_{n+1}$};
\filldraw [black] (2) circle (4pt);
\filldraw [black] (3) circle (4pt);
\filldraw [black] (5) circle (4pt);
\filldraw [black] (6) circle (4pt);
\draw[->] (3)--(6);
\draw[->] (5)--(6);
\end{tikzpicture}
\begin{tikzpicture}[scale=.5]
\coordinate (1) at (0,0);
\coordinate (2) at (.8,1.25);
\coordinate (3) at (2.2,1.25);
\coordinate (4) at (3,0);
\coordinate (5) at (2.2,-1.25);
\coordinate (6) at  (.8,-1.25);
\draw (2) arc (90:270:1.25);
\draw (3) arc (90:-90:1.25);
\draw (.5,1.55) node {$v_2$};
\draw (2.4,1.55) node {$v_1$};
\draw (2.4,-1.7) node {$v_{n+2}$};
\draw (.5,-1.7) node {$v_{n+1}$};
\filldraw [black] (2) circle (4pt);
\filldraw [black] (3) circle (4pt);
\filldraw [black] (5) circle (4pt);
\filldraw [black] (6) circle (4pt);
\draw[->] (3)--(6);
\draw[->] (2)--(3);
\end{tikzpicture}
\begin{tikzpicture}[scale=.5]
\coordinate (1) at (0,0);
\coordinate (2) at (.8,1.25);
\coordinate (3) at (2.2,1.25);
\coordinate (4) at (3,0);
\coordinate (5) at (2.2,-1.25);
\coordinate (6) at  (.8,-1.25);
\draw (2) arc (90:270:1.25);
\draw (3) arc (90:-90:1.25);
\draw (.5,1.55) node {$v_2$};
\draw (2.4,1.55) node {$v_1$};
\draw (2.4,-1.7) node {$v_{n+2}$};
\draw (.5,-1.7) node {$v_{n+1}$};
\filldraw [black] (2) circle (4pt);
\filldraw [black] (3) circle (4pt);
\filldraw [black] (5) circle (4pt);
\filldraw [black] (6) circle (4pt);
\draw[->] (2)--(5);
\draw[->] (5)--(6);
\end{tikzpicture}
\begin{tikzpicture}[scale=.5]
\coordinate (1) at (0,0);
\coordinate (2) at (.8,1.25);
\coordinate (3) at (2.2,1.25);
\coordinate (4) at (3,0);
\coordinate (5) at (2.2,-1.25);
\coordinate (6) at  (.8,-1.25);
\draw (2) arc (90:270:1.25);
\draw (3) arc (90:-90:1.25);
\draw (.5,1.55) node {$v_2$};
\draw (2.4,1.55) node {$v_1$};
\draw (2.4,-1.7) node {$v_{n+2}$};
\draw (.5,-1.7) node {$v_{n+1}$};
\filldraw [black] (2) circle (4pt);
\filldraw [black] (3) circle (4pt);
\filldraw [black] (5) circle (4pt);
\filldraw [black] (6) circle (4pt);
\draw[->] (2)--(5);
\draw[->] (2)--(3);
\end{tikzpicture}
\begin{tikzpicture}[scale=.5]
\coordinate (1) at (0,0);
\coordinate (2) at (.8,1.25);
\coordinate (3) at (2.2,1.25);
\coordinate (4) at (3,0);
\coordinate (5) at (2.2,-1.25);
\coordinate (6) at  (.8,-1.25);
\draw (2) arc (90:270:1.25);
\draw (3) arc (90:-90:1.25);
\draw (.5,1.55) node {$v_2$};
\draw (2.4,1.55) node {$v_1$};
\draw (2.4,-1.7) node {$v_{n+2}$};
\draw (.5,-1.7) node {$v_{n+1}$};
\filldraw [black] (2) circle (4pt);
\filldraw [black] (3) circle (4pt);
\filldraw [black] (5) circle (4pt);
\filldraw [black] (6) circle (4pt);
\draw[->] (2)--(5);
\draw[->] (3)--(6);
\end{tikzpicture}\\
\hline
3&
\begin{tikzpicture}[scale=.5]
\coordinate (1) at (0,0);
\coordinate (2) at (.8,1.25);
\coordinate (3) at (2.2,1.25);
\coordinate (4) at (3,0);
\coordinate (5) at (2.2,-1.25);
\coordinate (6) at  (.8,-1.25);
\draw (2) arc (90:270:1.25);
\draw (3) arc (90:-90:1.25);
\draw (.5,1.55) node {$v_2$};
\draw (2.4,1.55) node {$v_1$};
\draw (2.4,-1.7) node {$v_{n+2}$};
\draw (.5,-1.7) node {$v_{n+1}$};
\filldraw [black] (2) circle (4pt);
\filldraw [black] (3) circle (4pt);
\filldraw [black] (5) circle (4pt);
\filldraw [black] (6) circle (4pt);
\draw[->] (5)--(6);
\end{tikzpicture}
\begin{tikzpicture}[scale=.5]
\coordinate (1) at (0,0);
\coordinate (2) at (.8,1.25);
\coordinate (3) at (2.2,1.25);
\coordinate (4) at (3,0);
\coordinate (5) at (2.2,-1.25);
\coordinate (6) at  (.8,-1.25);
\draw (2) arc (90:270:1.25);
\draw (3) arc (90:-90:1.25);
\draw (.5,1.55) node {$v_2$};
\draw (2.4,1.55) node {$v_1$};
\draw (2.4,-1.7) node {$v_{n+2}$};
\draw (.5,-1.7) node {$v_{n+1}$};
\filldraw [black] (2) circle (4pt);
\filldraw [black] (3) circle (4pt);
\filldraw [black] (5) circle (4pt);
\filldraw [black] (6) circle (4pt);
\draw[->] (2)--(3);
\end{tikzpicture}
\begin{tikzpicture}[scale=.5]
\coordinate (1) at (0,0);
\coordinate (2) at (.8,1.25);
\coordinate (3) at (2.2,1.25);
\coordinate (4) at (3,0);
\coordinate (5) at (2.2,-1.25);
\coordinate (6) at  (.8,-1.25);
\draw (2) arc (90:270:1.25);
\draw (3) arc (90:-90:1.25);
\draw (.5,1.55) node {$v_2$};
\draw (2.4,1.55) node {$v_1$};
\draw (2.4,-1.7) node {$v_{n+2}$};
\draw (.5,-1.7) node {$v_{n+1}$};
\filldraw [black] (2) circle (4pt);
\filldraw [black] (3) circle (4pt);
\filldraw [black] (5) circle (4pt);
\filldraw [black] (6) circle (4pt);
\draw[->] (3)--(6);
\end{tikzpicture}
\begin{tikzpicture}[scale=.5]
\coordinate (1) at (0,0);
\coordinate (2) at (.8,1.25);
\coordinate (3) at (2.2,1.25);
\coordinate (4) at (3,0);
\coordinate (5) at (2.2,-1.25);
\coordinate (6) at  (.8,-1.25);
\draw (2) arc (90:270:1.25);
\draw (3) arc (90:-90:1.25);
\draw (.5,1.55) node {$v_2$};
\draw (2.4,1.55) node {$v_1$};
\draw (2.4,-1.7) node {$v_{n+2}$};
\draw (.5,-1.7) node {$v_{n+1}$};
\filldraw [black] (2) circle (4pt);
\filldraw [black] (3) circle (4pt);
\filldraw [black] (5) circle (4pt);
\filldraw [black] (6) circle (4pt);
\draw[->] (2)--(5);
\end{tikzpicture}\\
\hline
\end{tabular}
\end{center}
\caption{All possible $\sa _{n,n}$ with $1$ through $3$ middle edges removed.}
\label{table:G_2noptions}
\end{table}

\begin{lemma}For  $n\geq 3$ we have that
\begin{enumerate}
\item $[p_{(2n)}]X_{\san}=-3n^2+4n-2$ and
\item $[p_{(n^2)}]X_{\san}=2n-1$.
\end{enumerate}
\label{lem:powercoeffgraph}
\end{lemma}

\begin{proof}
To prove this we will use Lemma~\ref{lem:stanley} that considers all subsets $S$ of the edge set $E$. 
We are only interested in the subsets $S$ that yield $\lambda(S) = (2n)$ or $(n^2)$ both of which have parts  that are at least $n$. This means we will ignore any $S$ where $\lambda(S)$ has a part smaller than $n$ as these subsets will  not affect the coefficient of $p_{(2n)}$ or $p_{(n^2)}$ in $X_{\san}$. Note that if  $S$ has two or more edges removed from the left   $n$-path (or the right $n$-path), then $\lambda(S)$ certainly has a part  smaller than $n$. Thus, we will only consider subsets $S$ that have \emph{at most one edge} removed from the left $n$-path and \emph{at most one edge} removed from the right $n$-path. 
In Table~\ref{table:G_2noptions} we illustrate all possible graphs  with $1$ through $3$  middle edges removed from $\san$ since these will be central to our case analysis, consisting of 5 cases corresponding to the exclusion of 0 to 4 edges from $E$. 

First, consider $|S|=|E|$, so $S$ contains all the edges in $E$. This gives us the term $$(-1)^{2n+2}p_{(2n)}= p_{(2n)}.$$ 

Second, consider $|S|=|E|-1$, so $S$ has one fewer edge than $E$. Note that if we remove any one of the $2n+2$ edges from $\san$, then our graph is still connected. This gives us the term $$(-1)^{2n+1}(2n+2)p_{(2n)}= -(2n+2)p_{(2n)}.$$ 

Third, consider $|S|=|E|-2$, so $S$ has two fewer edges than $E$. If we exclude two edges from the middle, then we can see from Table~\ref{table:G_2noptions}  that all six possibilities result in connected graphs so we get the term $(-1)^{2n}6p_{(2n)}$. 

Instead we can exclude one edge  from the left $n$-path or right $n$-path and the other edge from the middle.  In  all four ways to  remove one edge  from  the middle, as illustrated in Table~\ref{table:G_2noptions}, we can also remove any one of the  $2(n-1)$ edges on the left $n$-path or right $n$-path and still have a connected graph. This gives us the term $(-1)^{2n}8(n-1)p_{(2n)}$. 

Finally, we could exclude no  edges from the middle, one of the $n-1$ edges from the left $n$-path, and one of the $n-1$ edges from the right $n$-path. Any of the $(n-1)^2$ choices  results in a  connected graph. This gives us the term $(-1)^{2n}(n-1)^2p_{(2n)}$.

Altogether from this case we have the term $$(n^2+6n-1)p_{(2n)}.$$

Fourth, consider $|S|=|E|-3$, where we exclude three edges from $E$. We can see from Table~\ref{table:G_2noptions} that excluding any three edges from  the middle leaves the graph  connected. This gives the term $(-1)^{2n-1}4p_{(2n)}$.  

If instead we remove two edges from the middle and one of the $n-1$ edges from the left $n$-path we can see from Table~\ref{table:G_2noptions} that only four out of six possibilities do not yield $\lambda(S)$ to have  a part smaller than $n$.  In these four possibilities  the graph is connected, which gives us the term $(-1)^{2n-1}4(n-1)p_{(2n)}$. We similarly get the term $(-1)^{2n-1}4(n-1)p_{(2n)}$ if we remove two  edges from the middle, one from the right $n$-path, and have all parts being at least $n$. 

Lastly, if we remove one edge from the middle, any one of the $n-1$ edges from the right $n$-path, and any one of the $n-1$ edges from the left $n$-path, then we can see from Table~\ref{table:G_2noptions} that the graph is still connected. This gives us the term $(-1)^{2n-1}4(n-1)^2p_{(2n)}$. 

Altogether from this case we get the term $$-4n^2p_{(2n)}.$$

Fifth and finally, consider $|S|=|E|-4$. If we exclude all four of the edges from the middle, then this disconnects our graph into two pieces of size $n$. The associated term is $(-1)^{2n-2}p_{(n^2)}$. 

Say we exclude three edges from the middle and one from the left $n$-path or right $n$-path. In any of these situations $\lambda(S)$ has a part smaller than $n$. 

Instead say that we remove two edges from the middle, one from the left $n$-path, and one from the right $n$-path. From  Table~\ref{table:G_2noptions} we can see that in only the  leftmost and rightmost pictures that  $\lambda(S)$ is not forced to have a part smaller than $n$. In the   leftmost and rightmost  pictures we will disconnect the graph into two pieces. For  any of the $n-1$ edges on the left $n$-path there is exactly one choice of an edge on the right $n$-path so that we break the graph into two pieces of size $n$. This gives the  term $(-1)^{2n-2}2(n-1)p_{(n^2)}$. 

Altogether from this case we get the term $$(2n-1)p_{(n^2)}.$$

Once we exclude more than four edges from $E$ we are guaranteed that $\lambda(S)$ will have  a part smaller  than $n$. Combining everything the coefficient of $p_{(2n)}$ is therefore
$$[p_{(2n)}]X_{\san}=1-(2n+2)+(n^2+6n-1)-4n^2=-3n^2+4n-2$$
and 
$$[p_{(n^2)}]X_{\san}=2n-1$$
is the  coefficient for $p_{(n^2)}$. 
\end{proof}

\begin{lemma} For $n\geq 1$ we have that 
\begin{enumerate}
\item $[e_{(n^2)}]p_{(2n)}=n$ and
\item $[e_{(n^2)}]p_{(n^2)}=n^2$. 
\end{enumerate}
\label{lem:elemcoeffpower}
\end{lemma}

\begin{proof}
Using Equation~\eqref{eq:newton} we have that $[e_{(n^2)}]p_{(2n)}=(-1)^{2n-2}\frac{2n(2-1)!}{2!}=n$. 

To prove the other coefficient note by Equation~\eqref{eq:newton} that $[e_{n}]p_{n}=(-1)^{n-1}\frac{n(1-1)!}{1!}=(-1)^{n-1}n$. In $p_{(n^2)}=p_np_n$ the coefficient of $e_{(n^2)}$ is purely determined by the multiplication of the coefficients of $e_{n}$ in  $p_n$, which gives $[e_{(n^2)}]p_{(n^2)}=[e_{n}]p_{n}[e_{n}]p_{n}=(-1)^{n-1}n(-1)^{n-1}n=n^2$. 
\end{proof}

We now apply these lemmas to determine the $e$-positivity of $X_{\san}$ for $n\geq 3$ in one final lemma.

\begin{lemma}\label{lem:sa_notepos}
The chromatic symmetric function of $\san$ for $n\geq 3$ is not $e$-positive. In particular, we have that
$$[e_{(n^2)}]X_{\san}=-n(n-1)(n-2).$$
\end{lemma}

\begin{proof}
By Observation~\ref{obs:newton} we can see that  $e_{(n^2)}$ has non-zero coefficient in $p_{\lambda}$ only for  $\lambda\vdash 2n$ with $\lambda\coarsep (n^2)$. There are only two partitions $\lambda\vdash 2n$ where $\lambda\coarsep (n^2)$, namely $(2n)$ and $(n^2)$. 

Since 
$$X_{\san}=\sum_{\lambda\vdash 2n}[p_{\lambda}]X_{\san} p_{\lambda}$$
the coefficient of $e_{(n^2)}$ in $X_{\san}$ only arises from the  $p_{(2n)}$ and $p_{(n^2)}$ terms. In particular, $$[e_{(n^2)}]X_{\san}=[e_{(n^2)}]p_{(2n)}\cdot [p_{(2n)}]X_{\san}+[e_{(n^2)}]p_{(n^2)}\cdot[p_{(n^2)}]X_{\san}.$$ Using  Lemma~\ref{lem:powercoeffgraph} and Lemma~\ref{lem:elemcoeffpower} we therefore have
\begin{align*}
[e_{(n^2)}]X_{\san}&=[e_{(n^2)}]p_{(2n)}\cdot [p_{(2n)}]X_{\san}+[e_{(n^2)}]p_{(n^2)}\cdot[p_{(n^2)}]X_{\san}\\
&=n\cdot (-3n^2+4n-2)+n^2\cdot (2n-1)\\
&=-n(n-1)(n-2).
\end{align*}
\end{proof}


We can now identify our first family of graphs that are \ccf, and whose chromatic symmetric functions are not $e$-positive.

\begin{theorem}\label{the:sanotpositive} The graphs $\san$ for all $n\geq 3$ are \ccf\ and not $e$-positive.
\end{theorem}

\begin{proof} This follows immediately from Lemmas~\ref{lem:sa_ccf} and \ref{lem:sa_notepos}.
\end{proof}

Since Theorem~\ref{the:sanotpositive} yields an infinite family of graphs with an even number of vertices that are \ccf\ and not $e$-positive, a natural question to ask is whether an infinite family of graphs exists with $N$ vertices, for all $N\geq 6$, which are \ccf\ and not $e$-positive. Such a family exists and is the family of \emph{augmented} saltire graphs, which we introduce now.

The \emph{augmented saltire graph} $\as _{a,b}$, where $a\geq 2, b\geq 3$, is the saltire graph $\sa _{a,b}$ with the additional edge $v_1v_{a+2}$. More precisely, $\as _{a,b}$, where $a\geq 2, b\geq 3$, is the graph on $a+b$  vertices $\{v_1, v_2, \ldots , v_{a+b}\}$ with edges $v_iv_{i+1}$ for $1\leq i \leq a+b-1$, $v_{a+b}v_1$, $v_1v_{a+1}$, $v_2v_{a+2}$ and $v_1v_{a+2}$. For example, $\as _{3,3}$,  $\as _{3,4}$ and a graphical representation of a generic $\as _{a,b}$ are given in Figure~\ref{fig:augsaltires}. 

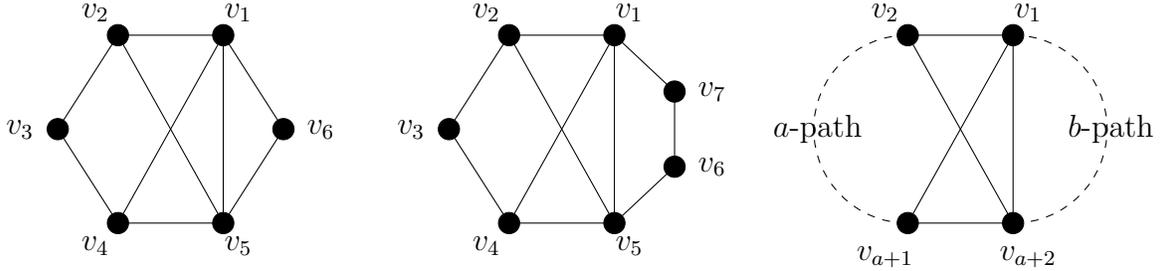
\begin{figure}[h]
\begin{center}
\begin{tikzpicture}[scale=1]
\begin{scope}[shift={(0,0)}]
\coordinate (1) at (0,0);
\coordinate (2) at (.8,1.25);
\coordinate (3) at (2.2,1.25);
\coordinate (4) at (3,0);
\coordinate (5) at (2.2,-1.25);
\coordinate (6) at  (.8,-1.25);
\draw (-.5,0) node {$v_3$};
\draw (.5,1.55) node {$v_2$};
\draw (2.4,1.55) node {$v_1$};
\draw (3.5,0) node {$v_6$};
\draw (2.4,-1.55) node {$v_5$};
\draw (.5,-1.55) node {$v_4$};
\filldraw [black] (1) circle (4pt);
\filldraw [black] (2) circle (4pt);
\filldraw [black] (3) circle (4pt);
\filldraw [black] (4) circle (4pt);
\filldraw [black] (5) circle (4pt);
\filldraw [black] (6) circle (4pt);
\draw[->] (2)--(5);
\draw[->] (3)--(6);
\draw[->] (3)--(5);
\draw[->] (1)--(2)--(3)--(4)--(5)--(6)--(1);
\end{scope}
\begin{scope}[shift={(5.2,0)}]
\coordinate (1) at (0,0);
\coordinate (2) at (.8,1.25);
\coordinate (3) at (2.2,1.25);
\coordinate (4) at (3,-.5);
\coordinate (5) at (2.2,-1.25);
\coordinate (6) at  (.8,-1.25);
\coordinate (7) at  (3,.5);
\draw (-.5,0) node {$v_3$};
\draw (.5,1.55) node {$v_2$};
\draw (2.4,1.55) node {$v_1$};
\draw (3.5,.5) node {$v_7$};
\draw (3.5,-.5) node {$v_6$};
\draw (2.4,-1.55) node {$v_5$};
\draw (.5,-1.55) node {$v_4$};
\filldraw [black] (1) circle (4pt);
\filldraw [black] (2) circle (4pt);
\filldraw [black] (3) circle (4pt);
\filldraw [black] (4) circle (4pt);
\filldraw [black] (5) circle (4pt);
\filldraw [black] (6) circle (4pt);
\filldraw [black] (7) circle (4pt);
\draw[->] (2)--(5);
\draw[->] (3)--(6);
\draw[->] (3)--(5);
\draw[->] (1)--(2)--(3)--(7)--(4)--(5)--(6)--(1);
\end{scope}
\begin{scope}[shift={(10.5,0)}]
\coordinate (1) at (0,0);
\coordinate (2) at (.8,1.25);
\coordinate (3) at (2.2,1.25);
\coordinate (4) at (3,0);
\coordinate (5) at (2.2,-1.25);
\coordinate (6) at  (.8,-1.25);
\draw[dashed] (2) arc (90:270:1.25);
\draw[dashed] (3) arc (90:-90:1.25);
\draw (-.4,0) node {$a$-path};
\draw (3.5,0) node {$b$-path};
\draw (.5,1.55) node {$v_2$};
\draw (2.4,1.55) node {$v_1$};
\draw (2.4,-1.7) node {$v_{a+2}$};
\draw (.5,-1.7) node {$v_{a+1}$};
\filldraw [black] (2) circle (4pt);
\filldraw [black] (3) circle (4pt);
\filldraw [black] (5) circle (4pt);
\filldraw [black] (6) circle (4pt);
\draw[->] (2)--(5);
\draw[->] (3)--(6);
\draw[->] (3)--(5);
\draw[->] (2)--(3);
\draw[->] (5)--(6);
\end{scope}
\end{tikzpicture}
\end{center}
\caption{From left to right we have $\as _{3,3}$, $\as _{3,4}$ and a generic $\as_{a,b}$.}
\label{fig:augsaltires}
\end{figure}

Using proofs very similar to those for saltire graphs, but with some additional cases generated by the edge $v_1v_{a+2}$, we can obtain the following two lemmas.

\begin{lemma}\label{lem:as_ccf} For all $a\geq 2, b\geq 3$ the graph $\as _{a,b}$ is \ccf. In particular, for $n\geq 3$ the graphs $\aseven$ and $\asodd$ are \ccf.
\end{lemma}

\begin{lemma}\label{lem:coeffG^*n,n+1} For $n\geq 3$ we have that
\begin{enumerate}
\item $[p_{(2n)}]X_{\aseven}= -4n^2+6n-2$,
\item $[p_{(n^2)}]X_{\aseven}= 3n-3$,
\item $[p_{(2n+1)}]X_{\asodd}= 4n^2-2n$ and
\item $[p_{(n+1,n)}]X_{\asodd}=-7n+4$.
\end{enumerate}
\end{lemma}

We also obtain the following lemma whose proof is analogous to that of Lemma~\ref{lem:elemcoeffpower}.

\begin{lemma}\label{lem:elemcoeffpower2} For $n\geq 1$ we have that
\begin{enumerate}
\item $[e_{(n+1, n)}]p_{(2n+1)}=-(2n+1)$ and
\item $[e_{(n+1, n)}]p_{(n+1,n)}= -n(n+1)$.
\end{enumerate}
\end{lemma}

From  here we are able to determine the $e$-positivity of $X_{\aseven}$ and $X_{\asodd}$ for $n\geq 3$.

\begin{lemma}\label{lem:as_notepos} 
The chromatic symmetric functions of ${\aseven}$ and ${\asodd}$ for $n\geq 3$ are not $e$-positive. In particular, we have that
$$[e_{(n^2)}]X_{\aseven}=[e_{(n+1,n)}]X_{\asodd}=-n(n-1)(n-2).$$
\end{lemma}

\begin{proof}
By Observation~\ref{obs:newton} we can see that  $e_{(n^2)}$ has non-zero coefficient in $p_{\lambda}$ only for  $\lambda\vdash 2n$ with $\lambda\coarsep (n^2)$. There are only two partitions $\lambda\vdash 2n$ where $\lambda\coarsep (n^2)$, namely $(2n)$ and $(n^2)$. 

Since 
$$X_{\aseven}=\sum_{\lambda\vdash 2n}[p_{\lambda}]X_{\aseven} p_{\lambda}$$
the coefficient of $e_{(n^2)}$ in $X_{\aseven}$ only arises from the  $p_{(2n)}$ and $p_{(n^2)}$ terms. In particular, $$[e_{(n^2)}]X_{\aseven}=[e_{(n^2)}]p_{(2n)}\cdot [p_{(2n)}]X_{\aseven}+[e_{(n^2)}]p_{(n^2)}\cdot[p_{(n^2)}]X_{\aseven}.$$ Using  Lemma~\ref{lem:coeffG^*n,n+1} and Lemma~\ref{lem:elemcoeffpower}  we therefore have
\begin{align*}
[e_{(n^2)}]X_{\aseven}&=[e_{(n^2)}]p_{(2n)}\cdot [p_{(2n)}]X_{\aseven}+[e_{(n^2)}]p_{(n^2)}\cdot[p_{(n^2)}]X_{\aseven}\\
&=n\cdot (-4n^2+6n-2)+n^2\cdot (3n-3)\\
&=-n(n-1)(n-2).
\end{align*}

Again by Observation~\ref{obs:newton} we can see that  $e_{(n+1,n)}$ has non-zero coefficient in $p_{\lambda}$ only for  $\lambda\vdash 2n+1$ with $\lambda \coarsep (n+1,n)$. There are only two partitions $\lambda\vdash 2n+1$ where $\lambda\coarsep (n+1,n)$, namely $(2n+1)$ and $(n+1,n)$. 

Since 
$$X_{\asodd}=\sum_{\lambda\vdash 2n+1}[p_{\lambda}]X_{\asodd} p_{\lambda}$$
the coefficient of $e_{(n+1,n)}$ in $X_{\asodd}$ only arises from the  $p_{(2n+1)}$ and $p_{(n+1,n)}$ terms. In particular, $$[e_{(n+1,n)}]X_{\asodd}=[e_{(n+1,n)}]p_{(2n+1)}\cdot [p_{(2n+1)}]X_{\asodd}+[e_{(n+1,n)}]p_{(n+1,n)}\cdot[p_{(n+1,n)}]X_{\asodd}.$$Using  Lemma~\ref{lem:coeffG^*n,n+1} and Lemma~\ref{lem:elemcoeffpower2}  we therefore have
\begin{align*}
[e_{(n+1,n)}]X_{\asodd}&=[e_{(n+1,n)}]p_{(2n+1)}\cdot [p_{(2n+1)}]X_{\asodd}+[e_{(n+1,n)}]p_{(n+1,n)}\cdot[p_{(n+1,n)}]X_{\asodd}\\
&=-(2n+1)\cdot (4n^2-2n)-n(n+1)\cdot (-7n+4)\\
&=-n(n-1)(n-2).
\end{align*}
\end{proof}

We can now identify our second family of graphs that are \ccf\ and whose chromatic symmetric functions are not $e$-positive, and, moreover, one such graph exists with $N$ vertices for all $N\geq 6$.

\begin{theorem}\label{the:asnotpositive} The graphs $\aseven$ and $\asodd$ for all $n\geq 3$ are \ccf\ and not $e$-positive.
\end{theorem}

\begin{proof} This follows immediately from Lemmas~\ref{lem:as_ccf} and \ref{lem:as_notepos}.
\end{proof}

\section{Triangular tower graphs}\label{sec:towers}

 {Our final section is devoted to answering the question of} does there exist a graph that is \ccf\ \emph{and} claw-free whose chromatic symmetric function is not $e$-positive? By exhaustive computational search the smallest such example is

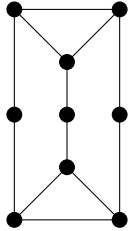
\begin{figure}[h]
\begin{center}
\begin{tikzpicture}[scale=.7]
\begin{scope}
\coordinate (1) at (-1,0);
\coordinate (2) at (0,1);
\coordinate (3) at (1,0);
\coordinate (4) at (0,3);
\coordinate (5) at (1,4);
\coordinate (6) at (-1,4);
\filldraw [black] (1) circle (4pt);
\filldraw [black] (2) circle (4pt);
\filldraw [black] (3) circle (4pt);
\filldraw [black] (4) circle (4pt);
\filldraw [black] (5) circle (4pt);
\filldraw [black] (6) circle (4pt);
\filldraw [black] (-1,2) circle (4pt);
\filldraw [black] (0,2) circle (4pt);
\filldraw [black] (1,2) circle (4pt);
\draw[->] (1)--(2);
\draw[->] (1)--(3);
\draw[->] (2)--(3);
\draw[->] (4)--(5);
\draw[->] (6)--(5);
\draw[->] (4)--(6);
\draw[->] (1)--(6);
\draw[->] (2)--(4);
\draw[->] (5)--(3);
\end{scope}
\end{tikzpicture}
\end{center}
\end{figure}

\noindent with chromatic symmetric function
\begin{align*}&12e_{(3,3,2,1)}-12e_{(3,3,3)}+102e_{(4,3,2)}+90e_{(4,4,1)}
+18e_{(5,2,2)}+96e_{(5,3,1)}\\&+294e_{(5,4)}+30e_{(6,2,1)}
+180e_{(6,3)}+342e_{(7,2)}+294e_{(8,1)}+666e_{(9)}\end{align*}and, moreover, it yields an infinite family of graphs that are \ccf, claw-free and not $e$-positive, the triangular tower graphs.

The \emph{triangular tower graph} $\trt _{a,b,c}$, where $a,b,c\geq 2$, is the graph on $a+b+c$ vertices
$$\{v_1,\ldots,v_{a}\}\cup\{v_{a+1},\ldots,v_{a+b}\}\cup\{v_{a+b+1},\ldots,v_{a+b+c}\}$$with edges
$v_iv_{i+1}$ for $$i\in \{1,\ldots ,a-1\}\cup\{a+1,\ldots, a+b-1\}\cup\{a+b+1,\ldots,a+b+c-1\}$$plus $\{v_1v_{a+1}, v_{a+1}v_{a+b+1}, v_{a+b+1}v_1\}$ and $\{v_av_{a+b}, v_{a+b}v_{a+b+c}, v_{a+b+c}v_a\}$. Informally we can visualize $\trt _{a,b,c}$ as consisting of three disjoint paths with, respectively, $a,b,c$ vertices where we take one leaf from each path and connect them in a triangle to form an induced $K_3$, and do the same with the remaining three leaves. For example, $\trt _{3,2,4}, \trt _{3,3,3}$ and a graphical representation of a generic $\trt _{a,b,c}$ are given in Figure~\ref{fig:tt}.

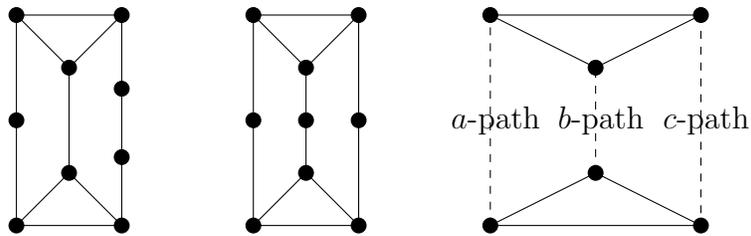
\begin{figure}[h]
\begin{center}
\begin{tikzpicture}[scale=.7]
\begin{scope}[shift={(0,0)}]
\coordinate (1) at (-1,0);
\coordinate (2) at (0,1);
\coordinate (3) at (1,0);
\coordinate (4) at (0,3);
\coordinate (5) at (1,4);
\coordinate (6) at (-1,4);
\filldraw [black] (1) circle (4pt);
\filldraw [black] (2) circle (4pt);
\filldraw [black] (3) circle (4pt);
\filldraw [black] (4) circle (4pt);
\filldraw [black] (5) circle (4pt);
\filldraw [black] (6) circle (4pt);
\filldraw [black] (-1,2) circle (4pt);
\filldraw [black] (1,1.3) circle (4pt);
\filldraw [black] (1,2.6) circle (4pt);
\draw[->] (1)--(2);
\draw[->] (1)--(3);
\draw[->] (2)--(3);
\draw[->] (4)--(5);
\draw[->] (6)--(5);
\draw[->] (4)--(6);
\draw[->] (1)--(6);
\draw[->] (2)--(4);
\draw[->] (5)--(3);
\end{scope}
\begin{scope}[shift={(4.5,0)}]
\coordinate (1) at (-1,0);
\coordinate (2) at (0,1);
\coordinate (3) at (1,0);
\coordinate (4) at (0,3);
\coordinate (5) at (1,4);
\coordinate (6) at (-1,4);
\filldraw [black] (1) circle (4pt);
\filldraw [black] (2) circle (4pt);
\filldraw [black] (3) circle (4pt);
\filldraw [black] (4) circle (4pt);
\filldraw [black] (5) circle (4pt);
\filldraw [black] (6) circle (4pt);
\filldraw [black] (-1,2) circle (4pt);
\filldraw [black] (0,2) circle (4pt);
\filldraw [black] (1,2) circle (4pt);
\draw[->] (1)--(2);
\draw[->] (1)--(3);
\draw[->] (2)--(3);
\draw[->] (4)--(5);
\draw[->] (6)--(5);
\draw[->] (4)--(6);
\draw[->] (1)--(6);
\draw[->] (2)--(4);
\draw[->] (5)--(3);
\end{scope}
\begin{scope}[shift={(10,0)}]
\coordinate (1) at (-2,0);
\coordinate (2) at (0,1);
\coordinate (3) at (2,0);
\coordinate (4) at (0,3);
\coordinate (5) at (2,4);
\coordinate (6) at (-2,4);
\filldraw [black] (1) circle (4pt);
\filldraw [black] (2) circle (4pt);
\filldraw [black] (3) circle (4pt);
\filldraw [black] (4) circle (4pt);
\filldraw [black] (5) circle (4pt);
\filldraw [black] (6) circle (4pt);
\draw[->] (1)--(2);
\draw[->] (1)--(3);
\draw[->] (2)--(3);
\draw[->] (4)--(5);
\draw[->] (6)--(5);
\draw[->] (4)--(6);
\draw[dashed] (1)--(6);
\draw[dashed] (2)--(4);
\draw[dashed] (5)--(3);
\draw (-1.9,2) node {$a$-path};
\draw (.1,2) node {$b$-path};
\draw (2.1,2) node {$c$-path};
\end{scope}
\end{tikzpicture}
\end{center}
\caption{From left to right we have $\trt_{3,2,4}$,  $\trt_{3,3,3}$ and a generic $\trt_{a,b,c}$.}
\label{fig:tt}
\end{figure}

Given $\trt _{a,b,c}$ we refer to the subgraphs induced by the edges
$$\{ v_iv_{i+1}\suchthat 1\leq i \leq a-1\}$$
as the \emph{$a$-path},
$$\{ v_iv_{i+1}\suchthat a+1\leq i \leq a+b-1\}$$
as the \emph{$b$-path} and
$$\{ v_iv_{i+1}\suchthat a+b+1\leq i \leq a+b+c-1\}$$
as the \emph{$c$-path}. Plus we refer to $\{v_1v_{a+1}, v_{a+1}v_{a+b+1}, v_{a+b+1}v_1\}$ as the \emph{top triangle}, and to $\{v_av_{a+b}, v_{a+b}v_{a+b+c}, v_{a+b+c}v_a\}$ as the \emph{bottom triangle}. 

As with the previous section we will focus on a subset of this family, namely $\trtn$. In this case we refer to the $a$-path, where $a=n$, as the \emph{left $n$-path}, the $b$-path, where $b=n$, as the \emph{middle $n$-path},  and the $c$-path, where $c=n$, as the \emph{right $n$-path}. We are now ready to ascertain the containment of the claw for this new family of graphs.

\begin{lemma}\label{lem:trt_ccf}  For all $a,b,c \geq 2$ the graph $\trt _{a,b,c}$ is \ccf\ and claw-free. In particular, for $n\geq 3$ the graph $\trtn$ is \ccf\ and claw-free.
\end{lemma}

\begin{proof} 
We first show that $\trt _{a,b,c}$ is claw-free by demonstrating a partition of the edges into disjoint sets such that every set edge induces a complete subgraph and no vertex belongs to more than two of the subgraphs. The result will then follow by Lemma~\ref{lem:beineke}. Note that such a partition is given by the edges in the top triangle, and the bottom triangle, edge inducing a $K_3$ subgraph each, and each remaining edge likewise edge inducing a $K_2$ subgraph.

Now we show that $\trt _{a,b,c}$ is \ccf\ by showing that the deletion of any three independent vertices from $\trt _{a,b,c}$ results in a disconnected graph. The result will then follow by Lemma~\ref{lem:BV}. Note that the removal of at least two non-adjacent vertices from either the $a$-path, $b$-path, or $c$-path results in a disconnected graph. Similarly the removal of one vertex from each of the $a$-path, the $b$-path, and the $c$-path of $\trt _{a,b,c}$ results in a disconnected graph unless all three vertices belong to the top triangle, or to the bottom triangle, but neither of these sets of three vertices is itself independent.
\end{proof}

Having proved that $\trt _{a,b,c}$ is both \ccf\ and claw-free we restrict our attention to $\trtn$ where $n\geq3$ and prove that its chromatic symmetric function is not $e$-positive by calculating the coefficient $[e_{(n^3)}] X_{\trtn}$. Note that $(3n),(2n,n)$ and $(n^3)$ are the only partitions that satisfy $\lambda \vdash 3n$ and $\lambda \coarsep (n^3)$ and hence by Obsevation~\ref{obs:newton} and Lemma~\ref{lem:stanley} in order to calculate $[e_{(n^3)}] X_{\trtn}$, we need to calculate $[p_{(3n)}]X_{\trtn}$, $[p_{(2n,n)}]X_{\trtn}$ and $[p_{(n^3)}]X_{\trtn}$, which we do in the following lemma using a case analysis that is similar to, but more substantial and delicate than,  Lemma~\ref{lem:powercoeffgraph}.


\begin{table}
\begin{center}
\begin{tabular}{| c |c|}
\hline
$i$& $\trtn$  with $i$ edges removed from the triangles\\
\hline
1&
6 of \hspace{.03in}
\begin{tikzpicture}[scale=.5]
\coordinate (1) at (-1,0);
\coordinate (2) at (0,1);
\coordinate (3) at (1,0);
\coordinate (4) at (0,2);
\coordinate (5) at (1,3);
\coordinate (6) at (-1,3);
\filldraw [black] (1) circle (4pt);
\filldraw [black] (2) circle (4pt);
\filldraw [black] (3) circle (4pt);
\filldraw [black] (4) circle (4pt);
\filldraw [black] (5) circle (4pt);
\filldraw [black] (6) circle (4pt);
\draw[->] (1)--(2);
\draw[->] (1)--(3);
\draw[->] (2)--(3);
\draw[->] (4)--(5);
\draw[->] (4)--(6);
\draw[->] (1)--(6);
\draw[->] (2)--(4);
\draw[->] (5)--(3);
\end{tikzpicture}\\
\hline
2& 
6 of  \hspace{.03in}
\begin{tikzpicture}[scale=.5]
\coordinate (1) at (-1,0);
\coordinate (2) at (0,1);
\coordinate (3) at (1,0);
\coordinate (4) at (0,2);
\coordinate (5) at (1,3);
\coordinate (6) at (-1,3);
\filldraw [black] (1) circle (4pt);
\filldraw [black] (2) circle (4pt);
\filldraw [black] (3) circle (4pt);
\filldraw [black] (4) circle (4pt);
\filldraw [black] (5) circle (4pt);
\filldraw [black] (6) circle (4pt);
\draw[->] (1)--(2);
\draw[->] (1)--(3);
\draw[->] (2)--(3);
\draw[->] (4)--(5);
\draw[->] (1)--(6);
\draw[->] (2)--(4);
\draw[->] (5)--(3);
\end{tikzpicture}, 
\hspace{.1in}3 of \hspace{.03in}
\begin{tikzpicture}[scale=.5]
\coordinate (1) at (-1,0);
\coordinate (2) at (0,1);
\coordinate (3) at (1,0);
\coordinate (4) at (0,2);
\coordinate (5) at (1,3);
\coordinate (6) at (-1,3);
\filldraw [black] (1) circle (4pt);
\filldraw [black] (2) circle (4pt);
\filldraw [black] (3) circle (4pt);
\filldraw [black] (4) circle (4pt);
\filldraw [black] (5) circle (4pt);
\filldraw [black] (6) circle (4pt);
\draw[->] (1)--(2);
\draw[->] (2)--(3);
\draw[->] (4)--(5);
\draw[->] (4)--(6);
\draw[->] (1)--(6);
\draw[->] (2)--(4);
\draw[->] (5)--(3);
\end{tikzpicture},
\hspace{.1in}6 of \hspace{.03in}
\begin{tikzpicture}[scale=.5]
\coordinate (1) at (-1,0);
\coordinate (2) at (0,1);
\coordinate (3) at (1,0);
\coordinate (4) at (0,2);
\coordinate (5) at (1,3);
\coordinate (6) at (-1,3);
\filldraw [black] (1) circle (4pt);
\filldraw [black] (2) circle (4pt);
\filldraw [black] (3) circle (4pt);
\filldraw [black] (4) circle (4pt);
\filldraw [black] (5) circle (4pt);
\filldraw [black] (6) circle (4pt);
\draw[->] (1)--(3);
\draw[->] (2)--(3);
\draw[->] (4)--(5);
\draw[->] (4)--(6);
\draw[->] (1)--(6);
\draw[->] (2)--(4);
\draw[->] (5)--(3);
\end{tikzpicture}\\
\hline
3&
12 of \hspace{.03in}
\begin{tikzpicture}[scale=.5]
\coordinate (1) at (-1,0);
\coordinate (2) at (0,1);
\coordinate (3) at (1,0);
\coordinate (4) at (0,2);
\coordinate (5) at (1,3);
\coordinate (6) at (-1,3);
\filldraw [black] (1) circle (4pt);
\filldraw [black] (2) circle (4pt);
\filldraw [black] (3) circle (4pt);
\filldraw [black] (4) circle (4pt);
\filldraw [black] (5) circle (4pt);
\filldraw [black] (6) circle (4pt);
\draw[->] (1)--(2);
\draw[->] (4)--(5);
\draw[->] (4)--(6);
\draw[->] (1)--(6);
\draw[->] (2)--(4);
\draw[->] (5)--(3);
\end{tikzpicture}, 
\hspace{.1in}6 of \hspace{.03in}
\begin{tikzpicture}[scale=.5]
\coordinate (1) at (-1,0);
\coordinate (2) at (0,1);
\coordinate (3) at (1,0);
\coordinate (4) at (0,2);
\coordinate (5) at (1,3);
\coordinate (6) at (-1,3);
\filldraw [black] (1) circle (4pt);
\filldraw [black] (2) circle (4pt);
\filldraw [black] (3) circle (4pt);
\filldraw [black] (4) circle (4pt);
\filldraw [black] (5) circle (4pt);
\filldraw [black] (6) circle (4pt);
\draw[->] (1)--(3);
\draw[->] (4)--(5);
\draw[->] (4)--(6);
\draw[->] (1)--(6);
\draw[->] (2)--(4);
\draw[->] (5)--(3);
\end{tikzpicture},
\hspace{.1in}2 of \hspace{.03in}
\begin{tikzpicture}[scale=.5]
\coordinate (1) at (-1,0);
\coordinate (2) at (0,1);
\coordinate (3) at (1,0);
\coordinate (4) at (0,2);
\coordinate (5) at (1,3);
\coordinate (6) at (-1,3);
\filldraw [black] (1) circle (4pt);
\filldraw [black] (2) circle (4pt);
\filldraw [black] (3) circle (4pt);
\filldraw [black] (4) circle (4pt);
\filldraw [black] (5) circle (4pt);
\filldraw [black] (6) circle (4pt);
\draw[->] (1)--(2);
\draw[->] (1)--(3);
\draw[->] (2)--(3);
\draw[->] (1)--(6);
\draw[->] (2)--(4);
\draw[->] (5)--(3);
\end{tikzpicture}\\
\hline
4&
6 of \hspace{.03in}
\begin{tikzpicture}[scale=.5]
\coordinate (1) at (-1,0);
\coordinate (2) at (0,1);
\coordinate (3) at (1,0);
\coordinate (4) at (0,2);
\coordinate (5) at (1,3);
\coordinate (6) at (-1,3);
\filldraw [black] (1) circle (4pt);
\filldraw [black] (2) circle (4pt);
\filldraw [black] (3) circle (4pt);
\filldraw [black] (4) circle (4pt);
\filldraw [black] (5) circle (4pt);
\filldraw [black] (6) circle (4pt);
\draw[->] (1)--(2);
\draw[->] (2)--(3);
\draw[->] (1)--(6);
\draw[->] (2)--(4);
\draw[->] (5)--(3);
\end{tikzpicture}, 
\hspace{.1in}6 of \hspace{.03in}
\begin{tikzpicture}[scale=.5]
\coordinate (1) at (-1,0);
\coordinate (2) at (0,1);
\coordinate (3) at (1,0);
\coordinate (4) at (0,2);
\coordinate (5) at (1,3);
\coordinate (6) at (-1,3);
\filldraw [black] (1) circle (4pt);
\filldraw [black] (2) circle (4pt);
\filldraw [black] (3) circle (4pt);
\filldraw [black] (4) circle (4pt);
\filldraw [black] (5) circle (4pt);
\filldraw [black] (6) circle (4pt);
\draw[->] (1)--(2);
\draw[->] (4)--(5);
\draw[->] (1)--(6);
\draw[->] (2)--(4);
\draw[->] (5)--(3);
\end{tikzpicture},
\hspace{.1in}3 of \hspace{.03in}
\begin{tikzpicture}[scale=.5]
\coordinate (1) at (-1,0);
\coordinate (2) at (0,1);
\coordinate (3) at (1,0);
\coordinate (4) at (0,2);
\coordinate (5) at (1,3);
\coordinate (6) at (-1,3);
\filldraw [black] (1) circle (4pt);
\filldraw [black] (2) circle (4pt);
\filldraw [black] (3) circle (4pt);
\filldraw [black] (4) circle (4pt);
\filldraw [black] (5) circle (4pt);
\filldraw [black] (6) circle (4pt);
\draw[->] (2)--(3);
\draw[->] (4)--(5);
\draw[->] (1)--(6);
\draw[->] (2)--(4);
\draw[->] (5)--(3);
\end{tikzpicture}\\
\hline
5&
6 of \hspace{.03in}
\begin{tikzpicture}[scale=.5]
\coordinate (1) at (-1,0);
\coordinate (2) at (0,1);
\coordinate (3) at (1,0);
\coordinate (4) at (0,2);
\coordinate (5) at (1,3);
\coordinate (6) at (-1,3);
\filldraw [black] (1) circle (4pt);
\filldraw [black] (2) circle (4pt);
\filldraw [black] (3) circle (4pt);
\filldraw [black] (4) circle (4pt);
\filldraw [black] (5) circle (4pt);
\filldraw [black] (6) circle (4pt);
\draw[->] (1)--(3);
\draw[->] (1)--(6);
\draw[->] (2)--(4);
\draw[->] (5)--(3);
\end{tikzpicture}\\
\hline
\end{tabular}
\end{center}
\caption{All possible $\trtn$ with $1$ through $5$ edges removed from the top and bottom triangle.}
\label{table:G_n,n,noptions}
\end{table}

\begin{lemma} For $n\geq 3$  we have that
\begin{enumerate}
\item $[p_{(3n)}]X_{\trtn}=(-1)^{3n+3}(12n^2-12n+2)$,
\item $[p_{(2n,n)}]X_{\trtn}=(-1)^{3n}(4n^2+6n-7)$ and
\item $[p_{(n^3)}]X_{\trtn}=(-1)^{3n-3}(3n-2)$.
\end{enumerate}
\label{lem:coeff_power_G_n,n,n}
\end{lemma}

\begin{proof}
To prove this we will use Lemma~\ref{lem:stanley} that considers all subsets of the edge set $E$. 
We are only interested in subsets $S\subseteq E$ that yield $\lambda(S)=(3n)$, $(2n,n)$ or $(n^3)$. Note that all of these have parts  at least $n$ so we will disregard any set $S$ where $\lambda(S)$ has a part smaller than $n$. If  $S$ has two or more edges removed from any of the $n$-paths, then $\lambda(S)$ certainly will have a part   smaller than $n$. Thus we will only consider subsets $S\subseteq E$ that have \emph{at most one edge} removed from any of the $n$-paths. 
In Table~\ref{table:G_n,n,noptions} we have considered all cases of $1$ through $5$ edges removed from the two triangles and have enumerated and collected all isomorphic graphs. This will be especially useful in our delicate case analysis, consisting of 7 cases corresponding to the removal of 0 to 6 edges from $E$. 

First, consider  $|S|=|E|$. This gives us the term $$(-1)^{3n+3}p_{(3n)}.$$ 

Second, consider $|S|=|E|-1$, and note that removing any one of the $3n+3$ edges yields a connected graph, and hence the term $$(-1)^{3n+2}(3n+3)p_{(3n)}.$$

Third, consider $|S|=|E|-2$. If the two removed edges come from the triangles there are  $15$ possibilities and we can see from Table~\ref{table:G_n,n,noptions} that all these possibilities are connected so contributes the term $(-1)^{3n+1}15p_{(3n)}$. 

Say we remove one edge from the triangles and one from the paths. In all 6 identical  possibilities of removing one edge from a triangle we can remove any one of the $3(n-1)$  edges from the $n$-paths and maintain a connected graph so we get the term  $(-1)^{3n+1}18(n-1)p_{(3n)}$. 

Next consider the situation where we remove two edges from the $n$-paths. We noted earlier that these two edges cannot be from the same $n$-path.  There are $\binom{3}{2}$ ways to choose the two $n$-paths and $n-1$ edge choices in each $n$-path. Since the resulting graph is always connected we have the term $(-1)^{3n+1}3(n-1)^2p_{(3n)}$. 

Altogether this case contributes the term $$(-1)^{3n+1}(3n^2+12n)p_{(3n)}.$$

Fourth, consider $|S|=|E|-3$. Now consider the situation of removing those three edges from the triangles. We can see from Table~\ref{table:G_n,n,noptions} that all $20$ possibilities are connected so contribute the term $(-1)^{3n}20p_{(3n)}$. 

Say instead we remove two edges from the triangles and one from the $n$-paths. In the 6 possibilities on the left in  Table~\ref{table:G_n,n,noptions} if we remove an edge from the left $n$-path we disconnect the graph yielding a part smaller than $n$. If instead we remove any one of the $2(n-1)$ edges from the middle or right $n$-paths, then we have a connected graph, which contributes the term $(-1)^{3n}12(n-1)p_{(3n)}$. In the remaining $9$ middle and right possibilities of removing two edges from the triangles  we can remove any one of the $3(n-1)$ edges from the three $n$-paths and still have a connected graph, which contributes the term $(-1)^{3n}27(n-1)p_{(3n)}$. 

Now say that we remove one edge from the triangle and two edges from the $n$-paths. Again, these two edges must be on different $n$-paths and any choice of edges on the two $n$-paths will leave the graph connected. With $6$ ways to remove an edge from the triangle, $\binom{3}{2}$ ways to choose the two $n$-paths, and $(n-1)^2$ ways to choose the edges on the $n$-paths we get the term $(-1)^{3n}18(n-1)^2p_{(3n)}$. 

Finally, consider the situation where we remove all three edges from the $n$-paths. No two of these removed edges are on the same $n$-path so we are removing one edge from each $n$-path. This will certainly disconnect the graph into two pieces. The only two-part partition we are interested in is $(2n,n)$ so we will count the edge removal choices that splits the graph yielding a partition of this type. We will first count the number of possibilities so that the piece connected to the top triangle has  $n$ vertices. Say  we remove an edge on the left $n$-path that results in $i$ vertices from this $n$-path contributing to this top connected piece. Also,  say we remove an edge from  the middle $n$-path so that the middle $n$-path contributes $j$ vertices to the top connected piece. As long as $1\leq i,j \leq n-1$ and $2\leq i+j\leq n-1$ then there exists exactly one edge in the right $n$-path that contributes $n-i-j$ vertices to the top connected piece, which yields our piece with $n$ vertices. The number of choices for $i$ and $j$ is  $\frac{(n-1)(n-2)}{2}$. Since there are equally many choices to instead make the bottom connected piece have $n$ vertices then this  contributes the term $(-1)^{3n}(n-1)(n-2)p_{(2n,n)}$. 

Altogether this case contributes the terms $$(-1)^{3n}(18n^2+3n-1)p_{(3n)}$$and $$(-1)^{3n}(n^2-3n+2)p_{(2n,n)}.$$ 

Fifth, consider $|S|=|E|-4$. There are $15$ possibilities for removing all four of the edges from the triangles. We can see in Table~\ref{table:G_n,n,noptions} that in $12$ of the possibilities the graph remains connected so contributes the term $(-1)^{3n-1}12p_{(3n)}$. In the remaining 3 possibilities the graph becomes an $n$-path and a $2n$-cycle so contributes the term $(-1)^{3n-1}3p_{(2n,n)}$. 

Next consider removing only three edges from the triangles and one edge from the $n$-paths. In the left $12$ possibilities listed in  Table~\ref{table:G_n,n,noptions} we can remove any of the $2(n-1)$ edges from the left and middle $n$-paths and maintain a connected graph.  The removal of any edge from the right $n$-path will yield a part smaller than $n$. In the middle $6$ possibilities listed in Table~\ref{table:G_n,n,noptions} we again can remove any one of the  $2(n-1)$ edges from the left or right $n$-path and maintain a connected graph, but choosing an edge from the middle $n$-path yields a part smaller than $n$.  In the right 2 possibilities in Table~\ref{table:G_n,n,noptions} any edge removed from any $n$-path would yield a part smaller than $n$ so altogether this contributes the term $(-1)^{3n-1}36(n-1)p_{(3n)}$. 

Next say we remove two edges from the triangles and two edges from the $n$-paths. In  Table~\ref{table:G_n,n,noptions} we can see for  the left $6$ possibilities that we can only remove the two edges from the  right and middle $n$-paths and this will split the graph into two pieces. Any choice of one of the $n-1$ edges from the middle $n$-path gives us precisely one choice for an edge in the right $n$-path so that we disconnect the graph to yield  $(2n,n)$. This contributes the term $(-1)^{3n-1}6(n-1)p_{(2n,n)}$. In the remaining 9 middle and right possibilities in Table~\ref{table:G_n,n,noptions} we can choose any two $n$-paths in $\binom{3}{2}$ ways and choose any edge in $(n-1)^2$ ways and still have a connected graph, which contributes the term $(-1)^{3n-1}27(n-1)^2p_{(3n)}$. 

Finally consider the situation when we remove only one edge from the triangles and three edges from the $n$-paths. Very similar to  earlier this breaks the graph into two pieces and there are $(n-1)(n-2)$ ways to choose the edges so that the graph is separated into one piece of size $2n$ and another of size $n$, which contributes the term $(-1)^{3n-1}6(n-1)(n-2)p_{(2n,n)}$. We cannot remove four edges from the $n$-paths else we yield a part smaller than $n$. 

Altogether this case contributes the terms $$(-1)^{3n-1}(27n^2-18n+3)p_{(3n)}$$and $$(-1)^{3n-1}(6n^2-12n+9)p_{(2n,n)}.$$

Sixth, consider $|S|=|E|-5$. If all five edges are removed from the triangles, then we have $6$ possibilities all of which give us a disconnected $n$-path and $2n$-path that contributes the term $(-1)^{3n-2}6p_{(2n,n)}$. 

Say we remove  four edges from the triangles and one from the $n$-paths. The left and middle 12 possibilities in  Table~\ref{table:G_n,n,noptions} will split the graph into two pieces, but not yielding the partition $(2n,n)$. In the right  3 possibilities in Table~\ref{table:G_n,n,noptions} we do not want to remove an edge from the left $n$-path since we would yield a part smaller than $n$, but we can remove any of the $2(n-1)$ other edges from the $n$-paths and  have the graph yield $(2n,n)$, which contributes the term $(-1)^{3n-2}6(n-1)p_{(2n,n)}$. 

Say we remove three edges from the triangles and two from the $n$-paths. In  Table~\ref{table:G_n,n,noptions} we can see with the right 2 possibilities that there is no choice of edges on the $n$-paths that  disconnects the graph yielding parts we are interested in. In the left and middle $18$ possibilities  there are two $n$-paths we can remove the two edges from without automatically yielding a part smaller than $n$. Also, any choice of edges on the $n$-paths splits the graph into two pieces so we need to count the possibilities that result in the partition $(2n,n)$. For any of the   $n-1$ choices for an edge on one $n$-path there is exactly one choice of an edge on the other $n$-path so we partition the graph  to yield $(2n,n)$. Together this contributes the term $(-1)^{3n-2}18(n-1)p_{(2n,n)}$. 

Finally, consider removing two edges from the triangles and three edges from the $n$-paths. 
 For the right and middle $9$ possibilities in Table~\ref{table:G_n,n,noptions} this splits the graph into two pieces and  there are $(n-1)(n-2)$ ways to choose the edges so that the graph is separated into a piece of size $2n$ and another of size $n$ as discussed earlier. In the left  $6$ possibilities in Table~\ref{table:G_n,n,noptions} we would split the graph so that it yields a part smaller than $n$. This contributes the term $(-1)^{3n-2}9(n-1)(n-2)p_{(2n,n)}$. Again, we cannot remove four or more edges from the $n$-paths else we yield a part smaller than $n$. 
 
Altogether this case gives us the term $$(-1)^{3n-2}(9n^2-3n)p_{(2n,n)}.$$

 {Seventh,} consider $|S|=|E|-6$. If we remove all six edges from the triangles we obtain   three disconnected $n$-paths, which contributes the term $(-1)^{3n-3}p_{(n^3)}$.

If we remove five edges from the triangles and one from the $n$-paths, then we can see that in all 6 possibilities in Table~\ref{table:G_n,n,noptions} we will split our graph into three pieces not yielding $(n,n,n)$. 

Say that we remove four edges from the triangles and two from the $n$-paths. In the left and middle 12 possibilities in Table~\ref{table:G_n,n,noptions} we will split the graph into pieces we are not interested in. In the right 3 possibilities in Table~\ref{table:G_n,n,noptions} we split the graph into three pieces and any one choice of the $n-1$ edges in the middle $n$-path will leave us with one choice of an edge in the right $n$-path so that we split our graph into three pieces of size $n$. This gives us the term $(-1)^{3n-3}3(n-1)p_{(n^3)}$. 

Say that we remove three edges from the triangles and three from the $n$-paths. In all 20 possibilities in Table~\ref{table:G_n,n,noptions} we split the graph into pieces of sizes we are not interested in. 

Since we cannot remove four or more edges from the $n$-paths and get pieces of size  at least $n$ altogether this case gives us the term $$(-1)^{3n-3}(3n-2)p_{(n^3)}.$$

No matter how we remove seven or more edges in total we will obtain a part smaller than $n$, so we have   considered all sets $S$ that contribute to the partitions $(3n)$, $(2n,n)$ and $(n^3)$. Adding everything  we get 
\begin{align*}[p_{(3n)}]X_{\trtn}&=(-1)^{3n+3}(1-(3n+3)+(3n^2+12n)-(18n^2+3n-1)+(27n^2-18n+3))\\&=(-1)^{3n+3}(12n^2-12n+2)\end{align*}and
\begin{align*}[p_{(2n,n)}]X_{\trtn}&=(-1)^{3n}((n^2-3n+2)-(6n^2-12n+9)+(9n^2-3n))\\&=(-1)^{3n}(4n^2+6n-7)\end{align*}and
\begin{align*}[p_{(n^3)}]X_{\trtn}&=(-1)^{3n-3}(3n-2).\end{align*}
\end{proof}

Similar to Lemma~\ref{lem:elemcoeffpower} we can prove the following.

\begin{lemma} For $n\geq 1$  we have that
\begin{enumerate}
\item $[e_{(n^3)}]p_{(3n)}=(-1)^{3n-3}n$,
\item $[e_{(n^3)}]p_{(2n,n)}=(-1)^{3n-3}n^2$ and
\item $[e_{(n^3)}]p_{(n^3)}=(-1)^{3n-3}n^3$.
\end{enumerate}
\label{lem:coeff_e[n,n,n]}
\end{lemma}

\begin{proof}
Using Equation~\eqref{eq:newton} we have that $[e_{(n^3)}]p_{(3n)}=(-1)^{3n-3}\frac{3n(3-1)!}{3!}=(-1)^{3n-3}n$. 

Recall from Lemma~\ref{lem:elemcoeffpower} and its proof that $[e_{(n^2)}]p_{(2n)}=(-1)^{2n-2}n$ and 
$[e_{n}]p_{n}=(-1)^{n-1}n$. In $p_{(n^3)}=p_np_np_n$ the coefficient of 
$e_{(n^3)}$ is purely determined by the multiplication of the coefficients of $e_{n}$ in  $p_n$, which gives 
$[e_{(n^3)}]p_{(n^3)}=([e_{n}]p_{n})^3=(-1)^{3n-3}n^3$. In $p_{(2n,n)}=p_{(2n)}p_n$ the coefficient of $e_{(n^3)}$ 
is purely determined by the multiplication of the coefficient of $e_{(n^2)}$ in  $p_{(2n)}$ and $e_{n}$ in  $p_{n}$, which gives $[e_{(n^3)}]p_{(2n,n)}=[e_{(n^2)}]p_{(2n)}[e_{n}]p_{n}=(-1)^{3n-3}n^2$. 
\end{proof}

\begin{lemma} The chromatic symmetric function of $\trtn$ for $n\geq3$ is not $e$-positive. In particular, we have that 
$$[e_{(n^3)}]X_{\trtn}=-n(n-1)^2(n-2).$$
\label{lem:trt_notepos}
\end{lemma}

\begin{proof}
By Observation~\ref{obs:newton} we can see that  $e_{(n^3)}$ has a non-zero coefficient in $p_{\lambda}$ only for  $\lambda\vdash 3n$ with $\lambda\coarsep (n^3)$. There are only three partitions $\lambda\vdash 3n$ where $\lambda\coarsep (n^3)$, namely $(3n)$, $(2n,n)$ and $(n^3)$. 

Since 
$$X_{\trtn}=\sum_{\lambda\vdash 3n}[p_{\lambda}]X_{\trtn} p_{\lambda}$$
the coefficient of $e_{(n^3)}$ in $X_{\trtn}$ only arises from the  $p_{(3n)}$, $p_{(2n,n)}$ and $p_{(n^3)}$ terms. In particular, 
$$[e_{(n^3)}]X_{\trtn}=[e_{(n^3)}]p_{(3n)}\cdot [p_{(3n)}]X_{\trtn}+[e_{(n^3)}]p_{(2n,n)}\cdot [p_{(2n,n)}]X_{\trtn}+[e_{(n^3)}]p_{(n^3)}\cdot[p_{(n^3)}]X_{\trtn}.$$
Using  Lemma~\ref{lem:coeff_power_G_n,n,n} and Lemma~\ref{lem:coeff_e[n,n,n]}  we therefore have
\begin{align*}
[e_{(n^3)}]X_{\trtn}&=n\cdot(12n^2-12n+2)-n^2\cdot(4n^2+6n-7)+n^3\cdot(3n-2)\\
&=-n(n-1)^2(n-2).
\end{align*}
\end{proof}


We can now identify our third family of graphs that are \ccf, are furthermore claw-free, and whose chromatic symmetric functions are not $e$-positive.

\begin{theorem}\label{the:trtnotpositive} The graphs $\trtn$ for all $n\geq 3$ are \ccf, claw-free and not $e$-positive.
\end{theorem}

\begin{proof} This follows immediately from Lemmas~\ref{lem:trt_ccf} and \ref{lem:trt_notepos}.
\end{proof}

\begin{remark}\label{rem:inc} One might ask whether triangular tower graphs are also incomparability graphs, so as to also potentially be a counterexample to Stanley's $(\mbox{\textbf{3}}+\mbox{\textbf{1}})$-free conjecture \cite[Conjecture 5.1]{Stan95}. However, it is straightforward to check that triangular tower graphs are not incomparability graphs.
\end{remark}

We conclude by conjecturing that the triangular tower graphs $\trtn$ for $n\geq 3$ are in some sense a minimal family of graphs that are \ccf, claw-free and whose chromatic symmetric functions are not $e$-positive. More precisely, we conjecture that there do not exist graphs with 10 or 11 vertices that are  \ccf, claw-free and whose chromatic symmetric functions are not $e$-positive. One motivation for this conjecture is the scarcity of graphs that are \ccf\ and whose chromatic symmetric function's expansion into elementary symmetric functions has negative coefficients. For $N=6$ only 4 of 112 connected graphs satisfy this. For $N=7$ this becomes 7 of 853 and for $N=8$ it is 27 of 11,117. Also the negative terms that we have identified are almost the \emph{only} negative terms in the elementary symmetric function expansion. For example, of the 293 terms in $X_{\trt _{7,7,7}}$ the only negative term is the one we identified, namely $-1260 e_{(7^3)}$. Of the 564 terms in $X_{\trt _{8,8,8}}$ the only negative term other than the one we identified is $-1944 e_{(4^6)}$, while of the 1042 terms in $X_{\trt _{9,9,9}}$ it is $-768 e_{(3^9)}$.

\bigskip
\footnotesize
\noindent\textit{Acknowledgments.}
The authors would like to thank Omar Antol\'{i}n Camarena and Jair Taylor for sharing their code, which we incorporated into ours when generating our extensive data sets in Sage \cite{Sage}. They would also like to thank Miki Racz for translating \cite[Theorem 1]{Krausz} and Richard Stanley for supportive conversations.  {Lastly, they would like to thank the referee for their very thoughtful and valuable suggestions that strengthened the paper.}

All three authors were supported  in part by the National Sciences and Engineering Research Council of Canada. This work was supported by the Canadian Tri-Council Research Support Fund.

\end{document}